\documentclass{ert-l}
\usepackage{amsthm,amssymb,amsmath,amscd}
\usepackage[all]{xy}
\usepackage{pb-diagram}
\newtheorem{thm}{Theorem}[section]

\newtheorem{lem}[thm]{Lemma}
\newtheorem{prop}[thm]{Proposition}

\newtheorem{rem}[thm]{Remark}

\theoremstyle{definition}

\newcommand{\Aset}{\mathbb{A}}

\newcommand{\Zset}{\mathbb{Z}}
\newcommand{\Rset}{\mathbb{R}}

\newcommand{\Bset}{\mathbb{B}}
\newcommand{\Cset}{\mathbb{C}}
\newcommand{\Iset}{\mathbb{I}}
\newcommand{\Lset}{\mathbb{L}}
\newcommand{\Mset}{\mathbb{M}}
\def\red{{\rm red}}

\newcommand{\Qset}{\mathbb{Q}}

\newcommand{\FF}{\mathbb{F}}

\newcommand{\uJ}{\underline{J}}
\newcommand{\underW}{\underline{W}}
\newcommand{\uP}{\bar{P}}
\newcommand{\uQ}{\bar{Q}}
\newcommand{\btau}{\boldsymbol{\tau}}
\newcommand{\bbs}{\mathbf{s}}
\newcommand{\bG}{\mathbf{G}}
\newcommand{\bT}{\mathbf{T}}

\newcommand{\bc}{\mathbf{c}}
\newcommand{\bd}{\mathbf{d}}
\newcommand{\be}{\mathbf{e}}

\newcommand{\bbb}{\mathbf{b}}
\newcommand{\bbi}{\mathbf{i}}

\def\iso{{\overset\sim\to\,}}

\def\PPhi{{\rm Frob}}

\def\rG{{{\rm G}}}
\def\rA{{{\rm A}}}

\def\Ima{{\rm Im}}
\def\Reel{{\rm Re}}
\def\cB{{\mathcal B}}

\def\cP{{\mathcal P}}

\def\Cent{{\rm Z}}
\def\val{{\rm val}}

\def\ie{{\it i.e.,\,}}

\def\cI{{\mathcal I}}
\def\cJ{{\mathcal J}}
\def\integers{{\mathfrak o}}
\def\ev{{\rm ev}}

\def\SL{{\rm SL}}

\def\SO{{\rm SO}}
\def\GL{{\rm GL}}
\def\PGL{{\rm PGL}}

\def\Cell{{\rm Cell}}

\def\Id{{\rm I}}

\def\II{{\rm I}}
\def\Ind{{\rm Ind}}

\def\id{{\rm e}}

\def\temp{{\rm t}}

\def\Mat{{\rm M}}
\def\Nor{{\rm N}}
\def\orb{{\mathcal C}}
\def\Irr{{\rm Irr}}

\def\aff{{\rm a}}

\def\cA{{\mathcal A}}

\def\cH{{\mathcal H}}

\def\cO{{\mathcal O}}

\def\cT{{\mathcal T}}
\def\cU{{\mathcal U}}
\def\cV{{\mathcal V}}
\def\cW{{\mathcal W}}

\def\aW{{W_{{\aff}}}}

\def\fs{{\mathfrak s}}

\def\fC{{\mathfrak C}}

\def\fii{{\mathfrak i}}

\def\integers{{\mathfrak o}}
\def\fR{{\mathfrak R}}
\def\fZ{{\mathfrak Z}}
\def\fX{{\mathfrak X}}

\def\CC{{\mathbb C}}
\def\bC{{\mathbb C}}
\def\bR{{\mathbb R}}

\def\bZ{{\mathbb Z}}
\def\fA{{\mathfrak A}}

\def\fR{{\mathfrak R}}
\def\fT{{\mathfrak T}}

\def\tW{{\widetilde W}}
\def\tmu{{\tilde \mu}}

\def\uT{{\bar T}}
\def\uX{{\bar X}}
\def\uW{{\bar W}}

\def\q{{/\!/}}
\def\nalpha{{\tau}}
\def\nbeta{{\tau'}}
\def\sgn{{\rm sgn}}
\begin{document}
\title[Geometric structure in the principal series]{Geometric structure in the principal series of the $p$-adic group $\rG_2$}
\author{Anne-Marie Aubert}
\address{Institut de Math\'ematiques de Jussieu, U.M.R. 7586 du C.N.R.S., Paris, France}
\email{aubert@math.jussieu.fr}
\author{Paul Baum}
\address{Pennsylvania State University, Mathematics Department, University Park, PA 16802, USA}
\email{baum@math.psu.edu}
\thanks{The second author was partially supported by NSF grant DMS-0701184}
\author{Roger Plymen}
\address{School of Mathematics, Alan Turing building, Manchester University, Manchester M13 9PL, England}
\email{plymen@manchester.ac.uk}


\date{}

\begin{abstract}In the representation theory of reductive $p$-adic groups
$G$, the issue of reducibility of induced representations is an
issue of great intricacy. It is our contention, expressed as a
conjecture in [3], that there exists a simple geometric structure
underlying this intricate theory.

We will illustrate here the conjecture with some detailed
computations in the principal series of $\rG_2$.

A feature of this article is the role played by cocharacters
$h_{\bc}$ attached to two-sided cells $\bc$ in certain extended
affine Weyl groups.

The quotient varieties which occur in the Bernstein programme are
replaced by extended quotients. We form the disjoint union
$\fA(G)$ of all these extended quotient varieties. We conjecture
that, after a simple algebraic deformation, the space $\fA(G)$ is
a model of the smooth dual $\Irr(G)$. In this respect, our
programme is a conjectural refinement of the Bernstein programme.

The algebraic deformation is controlled by the cocharacters
$h_{\bc}$. The cocharacters themselves appear to be closely
related to Langlands parameters.\end{abstract}

\maketitle

\bigskip

{\small
\textsc{Table of contents}
\begin{enumerate}
\item Introduction   \hfill 1 \item The strategy of the proof
\hfill 7 \item Background on the group $\rG_2$ \hfill 11 \item The
Iwahori point in $\fT(\rG_2)$ \hfill 13 \item Some preparatory
results \hfill 20 \item The two cases for which
$H^\fs=\GL(2,\Cset)$ \hfill 24 \item The case $H^\fs=\SL(3,\Cset)$
\hfill 27 \item The case $H^\fs=\SO(4,\Cset)$ \hfill 31 \item[]
References \hfill 43
\end{enumerate}
}

\section{Introduction}
In the representation theory of reductive $p$-adic
groups, the issue of reducibility of induced representations is an
issue of great intricacy: see, for example, the classic article by
Bernstein-Zelevinsky \cite{BZI} on $\GL(n)$ and the more recent
article by Mui\'c \cite{M} on $\rG_2$.
It is our contention, expressed as a conjecture in \cite{ABPCR}, that there
exists a simple geometric structure underlying this intricate theory.
We will illustrate here the conjecture with some
detailed computations in the principal series of $\rG_2$.

Let $F$ be a local nonarchimedean field, let $G$ be the group of
$F$-rational points in a connected reductive algebraic group
defined over $F$, and let $\Irr(G)$ be the set of equivalence
classes of irreducible smooth representations of $G$.

Our programme is a conjectural refinement of the Bernstein
programme, as we now explain.  Denote by $\fZ(G)$ the centre of
the category of smooth $G$-modules.  According to
Bernstein\cite{B}, the centre $\fZ(G)$ is isomorphic to the
product of finitely generated subalgebras, each of which is the
coordinate algebra of a certain irreducible algebraic variety, the
quotient $D/\Gamma$ of an algebraic variety $D$ by a finite group
$\Gamma$. Let $\Omega(G)$ denote the disjoint union of all these
quotient varieties.   The \emph{infinitesimal character} $inf.ch.$
is a finite-to-one map
\[
inf.ch.: \Irr(G) \to \Omega(G).\] Our basic idea is simple: we
replace each quotient variety $D/\Gamma$ by the extended quotient
$D\q \Gamma$, and form the disjoint union $\mathfrak{A}(G)$ of all
these extended quotient varieties.  We conjecture that, after a
simple algebraic deformation, the space $\mathfrak{A}(G)$ is a
model of the smooth dual $\Irr(G)$.

The algebraic deformation is controlled by finitely many
cocharacters $h_{\bc}$, one for each two-sided cell $\bc$ in the
extended affine Weyl group corresponding to $(D,\Gamma)$.
In fact, the cells
$\bc$ determine a decomposition of each extended quotient
$D\q\Gamma$.  The cocharacters themselves appear to be closely
related to Langlands parameters.

In this article, we verify the conjecture in \cite{ABPCR} for the
principal series of the  $p$-adic group $\rG_2$.  We have chosen
this example as a challenging test case.

\smallskip

Hence let $G=\rG_2(F)$ be the group of
$F$-points of a reductive algebraic group of type $\rG_2$. 
Let $T$ denote a maximal torus in $\rG_2$, and let $T^{\vee}$ denote
the dual torus in the Langlands dual $\rG_2^{\vee}$:\[ T^{\vee}
\subset \rG_2^{\vee} = \rG_2(\Cset).\] Since we are working with the
principal series of $\rG_2$, the algebraic variety $D$ has the
structure of the complex torus $T^{\vee}$.

Let $X(T^{\vee})$ denote the group of characters of $T^\vee$ and let
$X_*(T)$ denote the group of cocharacters of $T$. By duality,
these two groups are identified: $X_*(T) = X(T^{\vee})$. Let
$\Psi(T)$ denote the group of unramified characters of $T$. We
have an isomorphism
\[
T^{\vee} \cong \Psi(T), \quad \quad t \mapsto \chi_t\]where
\[
\chi_t(\phi(\varpi_F)) = \phi(t)\] for all $t \in T^{\vee}, \,\phi
\in X_*(T) = X(T^{\vee})$, and $\varpi_F$ is a uniformizer in $F$.

We consider pairs $(T,\lambda)$ consisting of a maximal torus $T$
of $G$ and a smooth quasicharacter $\lambda$ of $T$. Two such
pairs $(T_i,\lambda_i)$ are \emph{inertially equivalent} if there
exists $g \in G$ and $\psi \in \Psi(T_2)$ such that
\[
T_2 = T_1^g \quad \quad \text{and}\quad \quad \lambda_1^g =
\lambda_2 \otimes \psi.\] Here, $T_1^g = g^{-1}T_1g$ and
$\lambda_1^g: x \mapsto \lambda_1(gxg^{-1})$ for $x \in T_1^g$. We
write $[T,\lambda]_G$ for the inertial equivalence class of the
pair $(T,\lambda)$ and $\mathfrak{T}(G)$ for the set of all
inertial equivalence classes of the form $[T,\lambda]_G$.

We will choose a point $\fs \in \fT(G)$. Let $(T,\lambda) \in
\fs$. We will write
\[
D^{\fs}: = \{\lambda \otimes \psi: \psi \in \Psi(T)\}\]for the
$\Psi(T)$-orbit of $\lambda$ in $\Irr(T)$. Let $W(T)$ be the Weyl
group $\Nor_G(T)/T$. We set
\begin{equation} \label{Ws}
W^{\fs}:= \{w \in W(T)\,:\, w \cdot \lambda \in D^{\fs}\}.
\end{equation}

We have the standard projection
\[ \pi^\fs\colon D^{\fs}\q W^{\fs} \to D^{\fs}/W^{\fs}.\]

Section 1: this leads up to our main result, see Theorem
1.4.\smallskip

Section 2: we explain the strategy of our proof. 
\smallskip

Section 3: this  contains background material on $\rG_2$.\smallskip

Sections 4---8: these sections are devoted to our proof, which
requires $20$ Lemmas. The Lemmas are arranged in a logical
fashion: Lemma x.y is a proof of part $y$ of the conjecture for
the character $\lambda$ of $T$ which appears in section $x$. These
Lemmas involve some detailed representation theory, and some
calculations of the ideals $J_{\bc}$ in the asymptotic algebra $J$
of Lusztig corresponding to the complex Lie group $\SO(4,\Cset)$. 
The computation of the ideal $J_{\be_0}$ in section 8
is intricate. Our result here is especially interesting. We
establish a geometric equivalence (see definition below)
\begin{equation} \label{eqn:first}
J_{\be_o} \,\asymp \, \mathcal{O}(T^{\vee}/W^\fs) \oplus
\Cset\end{equation}
where $\be_0$ is the lowest two-sided cell and $\lambda=\chi\otimes\chi$
with $\chi$ a ramified quadratic character of $F^\times$. This
geometric equivalence has the effect of \emph{separating} the two
constituents of an $L$-packet in the principal series of $\rG_2$.

\smallskip

In \cite[\S 4]{ABP} we introduced a \emph{geometrical equivalence}
$\asymp$ between finite type $k$-algebras, which is generated by
elementary steps of the following three types: Morita equivalences, 
morphisms which are spectrum-preserving with respect to filtrations,
and deformations of central characters.
\emph{The assertion that $A \asymp B$ will mean that a definite geometrical 
equivalence has been constructed between $A$ and $B$.}
We would like to emphasize the fact that in~(\ref{eqn:first}) no
deformation of central character is used.

\smallskip

Let $\Irr(G)^{\fs}$ denote the $\fs$-component of $\Irr(G)$ in the
Bernstein decomposition of $\Irr(G)$. We will give the quotient
variety $T^{\vee}/W^\fs$ the Zariski topology, and $\Irr(G)^{\fs}$
the Jacobson topology. We note that irreducibility is an
\emph{open} condition, and so the set $\fR^\fs$ of reducible
points in $T^{\vee}/W^\fs$, i.e. those $(M,\psi\otimes\lambda)$
such that when parabolically induced to $G$, $\psi\otimes\lambda$
becomes reducible, is a sub-variety of $T^{\vee}/W^\fs$. The
reduced scheme associated to a scheme $\fX^\fs$ will be denoted
$\fX_{\red}^\fs$ as in \cite[p.25]{EH}. In the present context, a
\emph{cocharacter} will mean a homomorphism of algebraic groups
$\mathbb{C}^{\times} \to T^{\vee}$.

Let $\cH^\fs(G)$ be the Bernstein ideal of Hecke algebra of $G$
determined by $\fs \in \fT(G)$.
The point $\fs \in \fT(G)$ and the two-sided ideal $\cH^\fs(G)$
are said to be \emph{toral}.
Then $\fs$ is toral if and only if $D^\fs/W^\fs$ has maximal dimension
(that is, $2$ here) in $\Omega(G)$. 

Let $G^\vee=\rG_2(\Cset)$ be the complex reductive Lie group dual to
$G$, and let $T^\vee \subset \rG_2(\Cset)$. We define
\begin{equation} \label{eqn:Wsa}
\tW_\aff^\fs:=W^\fs\ltimes X(T^\vee).\end{equation}
The group $W^\fs$ is a \emph{finite Weyl group}
and $\tW_\aff^\fs$ is an \emph{extended affine Weyl group} (that
is, the semidirect product of an affine Weyl group by a
finite abelian group), see Section~\ref{sec:strategyproof}.
Let $\Phi^{\fs\vee}$ denote the coroot system of $W^\fs$,
and let $Y(T^\vee)$ denote the group of cocharacters of $T^\vee$.
Then let $H^\fs$ be the complex Lie group with root datum
$(X(T^\vee),\Phi^\fs,Y(T^\vee),\Phi^{\fs\vee})$.
We will see that the possible groups $H^\fs$ are $\rG_2(\Cset)$,
$\GL(2,\Cset)$, $\SL(3,\Cset)$ and $\SO(4,\Cset)$. We will consider these
cases in sections 4, 6, 7 and 8, respectively.

\smallskip
For simplicity, we shall assume in the Introduction that $H^\fs$ has
simply connected derived group.
This will be the case in
sections 4, 6 and 7 (not in section~8). 

As any extended affine Weyl group, the group $\tW_\aff^\fs$
is partitioned into \emph{two-sided cells}. This partition arises
(together with Kazhdan-Lusztig polynomials) from comparison of the
Kazhdan-Lusztig basis for the Iwahori-Hecke algebra of
$\tW_\aff^\fs$ with the standard basis. Let $\Cell(\tW_\aff^\fs)$
be the set of two-sided cells in $\tW_\aff^\fs$. The definition of
cells yields a natural partial ordering on $\Cell(\tW_\aff^\fs)$.
We will denote by $\bc_0$ the \emph{lowest} two-sided cell in $\tW_\aff^\fs$.

Let $J^\fs$ denote the Lusztig asymptotic algebra of the group
$\tW_\aff^\fs$ defined in \cite[\S 1.3]{LCellsIII}: this is a
$\Cset$-algebra whose structure constants are integers and which may be
regarded as an asymptotic version of the Iwahori-Hecke algebras
$\cH(\tW_\aff^\fs,\tau)$ of $\tW_\aff^\fs$, where $\tau\in\Cset^\times$.
Moreover, $J^\fs$ admits a canonical decomposition into finitely many
two-sided ideals
$J^\fs=\bigoplus J^\fs_\bc$, labelled by the the two-sided cells in
$\tW_\aff^\fs$.

\begin{prop} \label{prop:extpartition}
There exists a partition of $T^\vee\q W^\fs$ indexed by the two-sided cells in
$\tW_\aff^\fs$:
\[T^\vee\q W^\fs=\bigcup_{\bc\in\Cell(\tW_\aff^\fs)} (T^\vee\q W^\fs)_\bc.\]
The partition can be chosen so that the following property holds:
\begin{equation} \label{eqn:partitionq}
T^\vee/W^\fs\subset(T^\vee\q W^\fs)_{\bc_0}.
\end{equation}\end{prop}
\begin{rem} \label{rem:partition}
{\rm
The cell
decomposition in Proposition~\ref{prop:extpartition} is inherited from the
canonical decomposition of
the asymptotic algebra $J^\fs$ into two-sided ideals $J^\fs_\bc$:
see~(\ref{defn:extIwahori}), (\ref{defn:extGL}),
(\ref{defn:extSL}), (\ref{defn:extSO}), and Lemma~\ref{lem:extIwahori},
\ref{lem:extGL}, \ref{lem:extSL}, \ref{lem:extSO}.
We will also see there that the inclusion in~(\ref{eqn:partitionq}) can be strict.
}\end{rem}

We choose a partition
\[T^\vee\q W^\fs=\bigcup_{\bc\in\Cell(\tW_\aff^\fs)} (T^\vee\q W^\fs)_\bc,\]
so that~(\ref{eqn:partitionq}) holds.
We will call $(T^\vee\q W^\fs)_\bc$ the \emph{$\bc$-component} of
$T^\vee\q W^\fs$.

\smallskip

We will denote by $k$ the coordinate algebra $\cO(T^\vee/W^\fs)$ of the ordinary
quotient $T^\vee/W^\fs$. Then $k$ is isomorphic to the centre of
$\cH(\tW_\aff^\fs,\tau)$, \cite[\S 8.1]{LLuminy}. Let
\begin{equation} \label{eqn:Lusztigmap}
\phi_\tau\colon \cH(\tW_\aff^\fs,\tau)\to J^\fs
\end{equation}  be the $\Cset$-algebra
homomorphism that Lusztig defined in \cite[\S 1.4]{LCellsIII}. The centre of $J^\fs$
contains $\phi_\tau(k)$, see \cite[Prop.~1.6]{LCellsIII}. This provides
$J^\fs$ (and also each $J^\fs_\bc$) with a structure of $k$-module algebra
(this structure depends on the choice of $\tau$).
Both $\cH(\tW_\aff^\fs,\tau)$ and $J^\fs$ are finite type $k$-algebras.

\smallskip

We will assume that $p\ne 2,3,5$ in order to be able to apply the results of
Roche in \cite{Roc}. By combining \cite[Theorem~6.3]{Roc}
and \cite[Theorem~1]{ABP}, we obtain that the ideal $\cH^\fs(G)$
is a $k$-algebra Morita equivalent the $k$-algebra $\cH(\tW^\fs_\aff,q)$,
where $q$ is the order of the residue field of $F$.
On the other hand, the morphism $\phi_q\colon \cH(\tW_\aff^\fs,q)\to J^\fs$ is
spectrum-preserving with respect to filtrations, see
\cite[Theorem~9]{BN}.

It follows that the Bernstein ideal $\cH^\fs(G)$
is \emph{geometrically equivalent} to $J^\fs$ (which is 
equipped with the structure of $k$-algebra induced by $\phi_q$):
\begin{equation} \label{eqn:geHJ}
\cH^\fs(G)\asymp J^\fs.\end{equation}

\begin{rem} \label{rem:anyps}
{\rm We observe that similar arguments show that the geometrical equivalence
in~(\ref{eqn:geHJ}) is true for
each toral Bernstein ideal $\cH^\fs(G)$ of any $p$-adic group
$G$ (with the same restrictions on $p$ as in \cite[\S 4]{Roc}),
which is the group of $F$-points of an $F$-split connected reductive
algebraic group $\bG$ such the centre of $\bG$ is connected.
}\end{rem}

\begin{rem}\label{rem:sums}
{\rm We note that geometric equivalence respects direct sums. Suppose that we are given geometric equivalences
\[\alpha_1 : A_1 \asymp B_1, \quad \alpha_2 : A_2 \asymp B_2\]
We then have
\[
\alpha_1 \oplus 1: A_1 \oplus A_2 \asymp B_1 \oplus A_2,\quad \quad
1 \oplus \alpha_2: B_1 \oplus A_2 \asymp B_1 \oplus B_2\]
so that we have a definite geometrical equivalence
\[
(1 \oplus \alpha_2) (\alpha_1 \oplus 1) : A_1 \oplus A_2 \asymp B_1 \oplus B_2.\]
}
\end{rem}

By a case-by-case analysis for $G=\rG_2(F)$, we will prove that the
$k$-algebra $J^\fs$ (equipped here with the structure of $k$-algebra
induced by $\phi_1$) is itself geometrically equivalent
to the coordinate algebra $\cO(T^\vee\q W^\fs)$ of the extended quotient
$T^\vee\q W^\fs$:
\begin{equation} \label{eqn:geJextq}
J^\fs\,\asymp\,\cO(T^\vee\q W^\fs).\end{equation}
The geometric equivalence~(\ref{eqn:geJextq}) comes from
geometric equivalences
\begin{equation} \label{eqn:versetac}
J^\fs_\bc\,\asymp\,\cO((T^\vee\q
W^\fs)_\bc),
\quad\text{ for each $\bc\in\Cell(\tW_\aff^\fs)$}\end{equation}
that will be constructed in Lemma \ref{lem:extIwahori}, \ref{lem:extGL},
\ref{lem:extSL}, \ref{lem:extSO}.


\smallskip 

Let $\bc$ be a two-sided cell of $\tW_\aff^\fs$.
Then~(\ref{eqn:versetac}) induces a bijection
\begin{equation} \label{eqn:etac}
\eta_\bc^\fs\colon (T^\vee\q W^\fs)_\bc\to
\Irr(J^\fs_\bc).\end{equation}
We will denote by $\eta^\fs\colon T^\vee\q W^\fs\to \Irr(J^\fs)$ the bijection
defined by
\begin{equation} \label{eqn:eta}
\eta^\fs(t)=\eta^\fs_\bc(t), \quad\text{for $t\in (T^\vee\q W^\fs)_\bc$.}
\end{equation}
On the other hand, let
$\phi_{q,\bc}^\fs\colon \cH(\tW_\aff^\fs,q)\to J_\bc^\fs$ denote the
composition of the map $\phi_q^\fs$ and of the projection of $J$ onto
$J_{\bc}^\fs$. Let $E$ be a simple $J_{\bc}^\fs$-module, through the
homomorphism $\phi_{q,\bc}^\fs$, it is endowed with an
$\cH(\tW_\aff^\fs,q)$-module structure. We denote the
$\cH(\tW_\aff^\fs,q)$-module by $(\phi_{q,\bc}^\fs)^*(E)$. Lusztig
showed in \cite{LCellsIII} (see also \cite[\S 5.13]{Xi}) that the
set
\begin{equation}
\label{eqn:irrH}
\left\{(\phi_{q,\bc}^\fs)^*(E)\,:\,\begin{matrix}
\bc &\text{a two-sided cell of
$\tW_\aff^\fs$}\cr
E&\text{a simple $J_\bc^\fs$-module}
\end{matrix}\;\;\text{ (up to isomorphisms)}\right\}\end{equation}
is a complete set of simple $\cH(\tW_\aff^\fs,q)$-modules.

Hence we obtain a bijection
\begin{equation} \label{eqn:tmu}
\tmu^\fs\colon T^\vee\q
W^\fs\to\Irr(\cH(\tW_\aff^\fs,q))\end{equation} by setting
\begin{equation} \label{eqn:mu_c}
\tmu^\fs(t)=(\phi_{q,\bc}^\fs)^*\left(\eta_\bc^\fs(t)\right)\quad\text{
for $t\in (T^\vee\q W^\fs)_\bc.$}\end{equation} Let \begin{equation}
\label{eqn:mu} \mu^\fs\colon T^\vee\q
W^\fs\to\Irr(G)^\fs\end{equation} denote the composition of
$\tmu^\fs$ with the bijection \[\theta^\fs\colon
\Irr(\cH(\tW_\aff^\fs,q))\to\Irr(G)^\fs\] defined by Roche
\cite{Roc}.

We have
\begin{eqnarray}
 \mu^{\fs}: = \theta^\fs \circ (\phi_q^\fs)^* \circ \eta^\fs. \end{eqnarray}
We should emphasize that the map $\mu^{\fs}$ is not canonical: it
depends on a choice of geometrical equivalence inducing $\eta^\fs$.

\smallskip
\paragraph{Main result.}    Here is our main result, which is a consequence of $20$ Lemmas:
Lemma \ref{lem:extIwahori}, \ref{lem:familyIwahori},
\ref{lem:cocharacterIwahori},
\ref{lem:bijIwahori}, \ref{lem:tempIwahori}, \ref{lem:extGL}, \ref{lem:familyGL},
\ref{lem:cocharacterGL},  $\ldots$, \ref{lem:extSO}, \ref{lem:familySO},
\ref{lem:cocharacterSO}, \ref{lem:bijSO}, \ref{lem:tempSO}.

\begin{thm} \label{thm:main} Let $G=\rG_2(F)$ with $p\ne 2,3,5$.
Let $\fs = [T,\lambda]_G$ with $\lambda$ a smooth character of
\,$T$.  
Then we have
\begin{enumerate}
\item  The Bernstein ideal $\mathcal{H}^{\fs}(G)$ is geometrically equivalent to the coordinate algebra $\mathcal{O}(T^{\vee}  \q W^{\fs})$ of the affine variety
$T^{\vee} \q W^{\fs}$ and this determines a bijection
\[
\mu^{\fs} : T^{\vee} \q W^{\fs} \to \Irr(G)^{\fs}
\]
 \item
There is a flat family $\fX_\tau^\fs$ of subschemes of
$T^\vee/W^\fs$, with $\tau \in \Cset^{\times}$, such that
\[\fX_1^\fs = \pi^\fs(T^\vee\q W^\fs - T^\vee/W^\fs), \quad \quad
\fX^\fs_{\sqrt{q}} = \fR^\fs.\]The schemes $\fX_{1}^\fs$,
$\fX_{\sqrt q}^\fs$ are reduced. 
\item There exists a correcting system of cocharacters for the bijection
\[
\mu^{\fs}: T^{\vee} \q W^{\fs} \to \Irr(G)^{\fs}
\]
\item The geometrical equivalence in  1. can be chosen so that
\[(inf.ch.)\circ \mu^\fs = \pi_{\sqrt q}^\fs\] \item Let $E^{\fs}$
denote the maximal compact subgroup of $T^{\vee}$. Then we have
\[\mu^\fs(E^\fs\q W^\fs) = \Irr^{\temp}(G)^{\fs}.\]
\end{enumerate}
\end{thm}

\textbf{Correcting system of cocharacters.}   Recall (see \cite{ABP}) that, by definition, $T^{\vee} \q W^{\fs}$ is the quotient  
$\widetilde{T^{\vee}}/W^{\fs}$. Here 
\[
\widetilde{T^{\vee}} : = \{(w,t) \in W^{\fs} \times T^{\vee} : wt = t \}
\]
and 
\[
T^{\vee} \q W^{\fs} : = \widetilde{T^{\vee}}/W^{\fs}
\]
where $W^{\fs}$ acts on $\widetilde{T^{\vee}}$ by
\[
\alpha(w,t) = (\alpha w \alpha^{-1}, \alpha t)
\]
with $\alpha \in W^{\fs}, \; (w,t) \in \widetilde{T^{\vee}}$.
Then $\rho^{\fs}$ denotes the quotient map $\widetilde{T^{\vee}} \to T^{\vee} \q W^{\fs}$ and $\pi^{\fs}$ denotes the evident projection
\[
T^{\vee} \q W^{\fs} \to T^{\vee} / W^{\fs}, \quad \quad (w,t) \mapsto t.
\]

The extended quotient $T^{\vee} \q W^{\fs}$ is a complex affine variety and thus is the union of its irreducible components
$Z_1, Z_2, \ldots, Z_r$. A \emph{correcting system of cocharacters} for the bijection 
\[
\mu^{\fs} : T^{\vee} \q W^{\fs} \to \Irr(G)^{\fs}
\]
is an assignment to each $Z_j$ of a cocharacter $h_j$ of $T^{\vee}$. Each $h_j$ is a morphism of algebraic groups
\[
h_j \colon \CC^{\times} \to T^{\vee}.
\] 
If $Z_j$ and $Z_k$ are labelled by the same two-sided cell $\bc$ of $\widetilde{W^{\fs}_\aff}$  then $h_j = h_k$ and we will write
$h_j = h_k = h_{\bc}$. We have $h_{{\bc}_0} = 1$.  We require that once the finite sequence $h_1, h_2, \ldots, h_r$ has been fixed, it is possible to choose irreducible
components $X_1, X_2, \ldots, X_r$ of the affine variety $\widetilde{T^{\vee}}$ such that
\begin{itemize}
\item $\rho^{\fs} (X_j) = Z_j, \quad \quad j = 1,2,\ldots,r $.
\item For each $\tau \in \Cset^{\times}$, the map $X_j \to T^{\vee}/W^{\fs}$ which is the composition 
\[
X_j \to T^{\vee} \to T^{\vee}/W^{\fs}
\]
\[(w,t) \mapsto h_j(\tau)t \mapsto p^\fs(h_j(\tau)t)
\]
factors through $\rho^{\fs} \colon X_j \to Z_j$ and so gives a well-defined map
\[
\pi^{\fs}_{\tau} \colon Z_j  \to T^{\vee}/W^{\fs}
\]
with commutativity in the diagram
\[
\begin{CD}
X_j @>>> T^{\vee}/W^{\fs} \\
@V \rho^{\fs} VV             @VV id V \\
Z_j @>> \pi^{\fs}_{\tau} > T^{\vee}/W^{\fs}
\end{CD}
\]
Note that $h_j(\tau)t$ is the product in the algebraic group $T^{\vee}$ of $h_j(\tau)$ and $t$, and 
$p^\fs\colon T^{\vee} \to T^{\vee}/W^{\fs}$ is the standard quotient map.

\item Since $T^{\vee} \q W^{\fs}$ is the union of the $Z_j$, we have a morphism of affine varieties 
\[
\pi_{\tau}^{\fs} \colon T^{\vee} \q W^{\fs} \to T^{\vee}/W^{\fs}
\]
which we require to be a finite morphism with
\[
\pi_{\tau}^{\fs}(T^{\vee} \q W^{\fs} - T^{\vee}/W^{\fs}) = (\fX^{\fs}_{\tau})_{\red}
\]
and with
\[
(inf.ch.) \circ \mu^{\fs} = \pi^{\fs}_{\sqrt q}\, .
\]

\end{itemize}

\smallskip

\begin{rem} \label{rem:Conjsaretrue}
{\rm Theorem~\ref{thm:main} shows in particular that Conjecture~3.1 in
\cite{ABPCR} and part~(1) of Conjecture~1 in \cite{ABP} are both true for
the principal series of $\rG_2(F)$.
Moreover, we observe that the statement~(3) in Theorem~\ref{thm:main} is 
stronger than the statement~(2) in Conjecture~3.1  of \cite{ABPCR} in the sense
that cocharacters are dependent only on the two-sided cells, and not on
the irreducible components of $T^\vee\q W^\fs$ (in general,
$(T^\vee\q W^\fs)_\bc$ contains more than one irreducible component, see
Lemmas~\ref{lem:extIwahori}, \ref{lem:extGL} and \ref{lem:extSL}).
Also the bijection $\mu^\fs$ is not a homeomorphism in general (see
the Note after the proof of Lemma~\ref{lem:extSO})}.
\end{rem}

\section{The strategy of the proof} \label{sec:strategyproof}

Let $G=\rG_2(F)$ and let $\fs \in \fT(G)$. In this section we will
both explain the strategy of the proof of Theorem~1 and recall
some needed results from \cite{KL}, \cite{LCellsIV} and \cite{Re}.

\subsection{The extended affine Weyl group $\tW_\aff^\fs$}

It will follow from equations~(\ref{elements}), (\ref{eqn:Stwo}),
(\ref{eqn:StwoII}), (\ref{eqn:Stree}) and (\ref{eqn:four}),
that the possible groups $W^\fs$
are the dihedral group of order $12$, $\Zset/2\Zset$ (the cyclic group of
order $2$), the symmetric group $S_3$,
and the direct product $\Zset/2\Zset\times\Zset/2\Zset$.
In particular, $W^\fs$ is a finite Weyl group.
Let $\Phi^\fs$ denote its root system, and let
$Q^\fs=\Zset\Phi^\fs$ be the corresponding root lattice.

The group $W^{\fs}_\aff:=W^\fs\ltimes Q^\fs$ is an \emph{affine Weyl
group}. Let $S^\fs$ be a set of simple reflections of $W^{\fs}_\aff$ such
that $S^\fs\cap W^\fs$ generates $W^\fs$ and is a set of simple
reflections of $W^\fs$. Then one can find an abelian subgroup $C^\fs$ of
$\tW^\fs_\aff$ such that $cS^\fs=S^\fs c$ for any $c\in C^\fs$ and we have
$\tW^\fs_\aff=W^\fs_\aff\ltimes C^\fs$. This shows that
$\tW_\aff^\fs$ is an \emph{extended affine Weyl group}.
Let $\ell$ denote the length function on $W^{\fs}_\aff$. We extend $\ell$
to $\tW_\aff^\fs$ by $\ell(wc)=\ell(w)$, if $w\in W^{\fs}_\aff$ and $c\in
C$.

\subsection{The Iwahori-Hecke algebra $\cH(\tW^\fs_\aff,\tau)$}
Let $\btau$ be an indeterminate and let $\cA=\Cset[\btau,\btau^{-1}]$.
Let $\cW\in\{W^{\fs}_\aff,\tW^{\fs}_\aff\}$.
We denote by $\cH(\cW,\btau^2)$ the \emph{generic Iwahori-Hecke algebra} of $\cW$,
that is, the free $\cA$-module with basis $(T_w)_{w\in \cW}$ and multiplication
defined by the relations
\begin{eqnarray}
T_wT_{w'}=T_{ww'},&\text{ if
$\ell(ww')=\ell(w)+\ell(w')$,}\label{eqnarray:lineone}\\
(T_s-\btau)(T_s+\btau^{-1})=0,&\text{if $s\in S^\fs$.}\label{eqnarray:linetwo}\end{eqnarray}
The Iwahori-Hecke algebra $\cH(\cW,\tau^2)$ associated to
$(\cW,\tau)$ with $\tau\in\Cset^\times$ is obtained from $\cH(\cW,\btau^2)$
by specializing $\btau$ to $\tau$, that is, it is the algebra generated by
$T_w$, $w\in \cW$, with
relations~(\ref{eqnarray:lineone}) and the analog of~(\ref{eqnarray:linetwo})
in which $\btau$ has been replaced by $\tau$.

We observe that the works of Reeder \cite{Re} and Roche \cite{Roc} reduce
the study of $\Irr^\fs(G)$ to those of the simple modules of
$\cH(\tW^\fs_\aff,q)$. A classification of these simple modules by
\emph{indexing triples} $(t,u,\rho)$ is
provided by \cite{KL} and \cite{Re}. We will recall some features of this
classification in the next subsection.

\subsection{The indexing triples}
\label{subsec:indextriples}

Let $I_F$ be
the inertia subgroup of the Weil group $W_F$, let $\PPhi_F
\in W_F$ be a geometric Frobenius (a generator of $W_F/I_F=\Zset$), and
let $\Phi$ be a Langlands parameter (that is, the
equivalence class modulo inner automorphisms by elements of $H^\fs$ of an
homomorphism from $W_F \times
\SL(2,\mathbb{C})$ to $H^\fs$ which is
admissible in the sense of \cite[\S~8.2]{Bo}).
We assume that $\Phi$ is \emph{unramified}, that is, that $\Phi$ is
trivial on $I_F$.
We will still denote by $\Phi$ the restriction of $\Phi$ to $\SL(2,\Cset)$.

Let $u$ be the unipotent element of $H^\fs$ defined by
\begin{equation} \label{eqn:Phiu}
u = \Phi\left( \begin{array}{cc}
1 & 1 \\
0 & 1 \end{array} \right).\end{equation}
We set
$t=\Phi(\PPhi_F)$. Then $t$ is a semisimple element in $H^\fs$ which commutes
with $u$.
Let $\Cent_{H^\fs}(t)$ denote the centralizer of $t$ in $H^\fs$ and let
$\Cent_{H^\fs}^\circ(t)$ be the identity connected component
of $\Cent_{H^\fs}(t)$. We observe that if $H^\fs$ is one of the groups
$\rG_2(\Cset)$, $\GL(2,\Cset)$, $\SL(3,\Cset)$, then
$\Cent_{H^\fs}(t)$ is always connected.

For each $\tau\in\Cset^\times$, we set \[t(\tau)=\Phi \left(
\begin{array}{cc}
\tau & 0 \\
0 & \tau^{-1}\end{array} \right)\in\Cent_{H^\fs}^\circ(t).\]

Lusztig constructed in \cite[Theorem~4.8]{LCellsIV}
a bijection $\cU\mapsto \bc(\cU)$ between the set of
unipotent classes in $H^\fs$ and the set of
two-sided cells of $\tW_\aff^\fs$.
Let $\bc$ be the two-sided cell of
$\tW^\fs_\aff$ which corresponds by this bijection to the unipotent class
to which $u$ belongs
and then let the $L$-parameter $\Phi$ be such that
(\ref{eqn:Phiu}) holds.
We will denote by
$F_\bc$ the centralizer in $H^\fs$ of $\Phi(\SL(2,\Cset))$; then
$F_\bc$ is a maximal reductive subgroup of $\Cent_{H^\fs}(u)$.

For any $\tau\in\Cset^\times$, we set:
\[h_{\bc}(\tau): = \Phi \left(
\begin{array}{cc}
\tau & 0 \\
0 & \tau^{-1}\end{array} \right).\]
Since $\Phi$ is only determined up to conjugation, we can always assume
that $h_{\bc}(\tau)$ belongs to $T^\vee$. Then it defines a cocharacter \[h_{\bc}\colon
\mathbb{C}^{\times} \to T^{\vee}.\]
The element \begin{equation} \label{eqn:sigmat}
\sigma: = h_{\bc}(\sqrt q)\cdot t\end{equation}
satisfies the equation
\begin{equation} \label{eqn:indtriple}
\sigma u\sigma^{-1} = u^q.
\end{equation}

For $\sigma$ a semisimple element in $H^\fs$ and $u$ a unipotent
element in $H^\fs$ such that~(\ref{eqn:indtriple}) holds,
let $\cB^\fs_{\sigma,u}$ be the variety of Borel subgroups of
$H^\fs$ containing $\sigma$ and $u$, and let $A_{\sigma,u}$ be the component
group of the simultaneous centralizer of $\sigma$ and $u$ in $H^\fs$.
Let $\cT(H^\fs)$ denote the set of triples $(\sigma,u,\rho)$ such that
$\sigma$ is a semisimple element in $H^\fs$, $u$ is a unipotent
element in $H^\fs$ which satisfy~(\ref{eqn:indtriple}), and $\rho$ is an irreducible
representation of $A_{\sigma,u}$ such that $\rho$ appears in the natural
representation of $A_{\sigma,u}$ on $H^*(\cB^\fs_{\sigma,u},\Cset)$.

Reeder proved in \cite{Re}, using the construction of Roche \cite{Roc}, that
the set $\Irr^\fs(G)$ is in bijection with the $H^\fs$-conjugacy classes
of triples $(\sigma,u,\rho)\in\cT(H^\fs)$.
The irreducible $G$-module corresponding to the
$H^\fs$-conjugacy class of $(\sigma,u,\rho)$
will be denoted $\cV_{\sigma,u,\rho}^\fs$ and we will refer to
the triples $(\sigma,u,\rho)$ as \emph{indexing triples} for
$\Irr^\fs(G)$.

\medskip

\subsection{The $L$-parameters}
\label{subsec:Lparameters}

Let $W_F^a$ be the topological abelianization of $W_F$ and let $I_F^a$ be
the image in $W_F^a$ of the inertia subgroup $I_F$.
We denote by \[r_F\colon W^a_F\to F^\times\] the reciprocity isomomorphism
of abelian class-field theory, and set $\varpi_F:=r_F(\PPhi_F)$ a prime
element in $F$.
Then the map $x\mapsto x(\varpi_F)$ defines an embedding of $X(T^\vee)=
Y(T)$ in the $p$-adic torus $T$. This embedding gives a splitting
$T=T(\integers_F)\times X(T^\vee)$.

We assume given $(t,u)$ as above, \ie $t=\Phi(\PPhi_F)$ and $u$
satisfies~(\ref{eqn:Phiu}).

Let $\integers_F$ denote the ring of integers of $F$,
let $\lambda^\circ$ be an irreducible character of $T(\integers_F)$, and let
$\lambda$ be an extension of $\lambda^0$ to $T$.
Let \[\hat\lambda\colon I^a_F\to T^\vee\] be the unique homomorphism
satisfying
\begin{equation} \label{eqn:Langlandstores}
x\circ\hat\lambda=\lambda^\circ \circ x\circ r_F,\quad\text{for $x\in
X(T^\vee)$,}\end{equation}
where $x$ is viewed as in $X(T^\vee)$ on the left side
of~(\ref{eqn:Langlandstores}) and as an element of $Y(T)$ (a cocharacter
of $T$) on the right side.

The choice of Frobenius $\PPhi_F$ determines a splitting
\[W_F^a=I_F^a\times\langle \PPhi_F\rangle,\]
so we can extend $\hat\lambda$ to a homomorphism $\hat\lambda_t\colon
W_F^a\to G^\vee$ by setting
\[\hat\lambda_t(\PPhi_F):=t.\]
Then we define (see \cite[\S~4.2]{Re}):
\[\tilde\Phi\colon W_F^a\times\SL(2,\CC)\to G^\vee\quad
(w,m)\mapsto\hat\lambda_t(w)\cdot\Phi(m).\]

\subsection{The asymptotic Hecke algebra $J^\fs$}
Let $\bar{}\colon \cA\to \cA$ be
the ring involution which takes $\btau^n$ to $\btau^{-n}$ and let
$h\mapsto\bar{h}$ be the unique endomorphism of $\cH(W^{\fs}_\aff,\btau)$
which is $\cA$-semilinear with respect to
$\bar{}\colon \cA\to \cA$ and
satisfies $\bar{T}_s=T_s^{-1}$ for any $s\in S^\fs$.
Let $w\in W^{\fs}_\aff$. There is a unique
\[C_w\in \bigoplus_{w\in W^{\fs}_\aff}\Zset[\btau^{-1}]T_w\quad
\text{such that}\]
\[\bar C_w=C_w\;\;\text{ and }\;\;C_w=T_w\pmod{\bigoplus_{y\in W^{\fs}_\aff}
(\bigoplus _{m<0}\Zset \btau^m)T_y}\]
(see for instance \cite[Theorem~5.2~(a)]{Lbook}).
We write \[C_w=\sum_{y\in W^{\fs}_\aff}P_{y,w}\,T_y,\quad\text{where
$P_{y,w}\in \Zset[\btau^{-1}]$.}\]
For $y\in W^{\fs}_\aff$, $c,c'\in C^\fs$, we define $P_{yc,wc'}$
as $P_{y,w}$ if $c=c'$ and as $0$ otherwise. Then for $w\in \tW_\aff^\fs$, we
set
$C_w=\sum_{y\in \tW_\aff^\fs}P_{y,w}\,T_y$.
It follows from \cite[Theorem~5.2~(b)]{Lbook} that
$(C_w)_{w\in \tW_\aff^\fs}$ is an $\cA$-basis of $\cH(\tW_\aff^\fs,\btau)$.
For $x,y,w$ in $W$, let $h_{x,y,w}\in \cA$ be defined by
\[C_x\cdot C_y=\sum_{w\in \tW_\aff^\fs}h_{x,y,w}\,C_w.\]
For any $w\in \tW_\aff^\fs$, there exists a non negative
integer $a(w)$ such that
\begin{eqnarray*}
h_{x,y,w}\in& \btau^{a(w)}\,
\Zset[\btau^{-1}]& \text{for all $x,y\in \tW_\aff^\fs$,}\\
h_{x,y,w}\notin &\btau^{a(w)-1}\,\Zset[\btau^{-1}]&\text{for some $x,y\in
\tW_\aff^\fs$.}\end{eqnarray*}
Let $\gamma_{x,y,w^{-1}}$ be the coefficient of $\btau^{a(w)}$ in
$h_{x,y,w}$.

Let $\uJ^\fs$ denote the free Abelian group with basis $(t_w)_{w\in
\tW^\fs_\aff}$. Lusztig has defined an associative ring structure
on $\uJ^\fs$ by setting
\[t_x\cdot t_y:=\sum_{w\in \tW^\fs_\aff}\gamma_{x,y,w^{-1}}\,t_w.
\quad\text{(This is a finite sum.)}\] The ring $\uJ^\fs$ is called the \emph{based ring} of
$\tW^\fs_\aff$. It has a unit element.
The $\Cset$-algebra \begin{equation} \label{eqn:Js}
J^\fs:=\uJ^\fs\otimes_\Zset\Cset\end{equation} is called the
\emph{asymptotic Hecke algebra} of $\tW^\fs_\aff$.

For each two-sided cell $\bc$ in $\tW^\fs_\aff$, the subspace
$\uJ_\bc^\fs$ spanned by the $t_w$, $w\in\bc$, is a two-sided ideal of
$\uJ^\fs$.
The ideal $\uJ_\bc^\fs$ is an associative ring, with a unit,
which is called the based
ring of the
two-sided cell $\bc$ and
\begin{equation} \label{deuJ}
\uJ^\fs=\bigoplus_{\bc\in\Cell(\tW_\aff^\fs)} \uJ^\fs_\bc
\end{equation}
is a direct sum decomposition of $\uJ^\fs$ as a ring.
We set $J^\fs_\bc:=\uJ^\fs_\bc\otimes_\Zset\Cset$.

\section{Background on the group $\rG_2$}

Let $\bG=\rG_2$ be a group of type $\rG_2$ over a commutative field
$\FF$, and let $\rG_2(\FF)$ denote its group of $\FF$-points.

\subsection{Roots and fundamental reflexions}
Denote by $\bT$ a maximal split torus in $\bG$, and by $\Phi$ the set of
roots of $\bG$ with respect to $\bT$. Let
$(\varepsilon_1,\varepsilon_2,\varepsilon_3)$ be the canonical
basis of $\Rset^3$, equipped with the scalar product $(\;|\;)$ for
which this basis is orthonormal. Then
$\alpha:=\varepsilon_1-\varepsilon_2$,
$\beta:=-2\varepsilon_1+\varepsilon_2+\varepsilon_3$ defines a
basis of $\Phi$ and
\[\Phi^+=\left\{\alpha,\beta,\alpha+\beta,2\alpha+\beta,3\alpha+\beta,
3\alpha+2\beta\right\}\] is a subset of positive roots in $\Phi$
(see \cite[Planche IX]{Bou}).

We have
\begin{equation} \label{products}
(\alpha|\alpha)=2,\;\;\; (\beta|\beta)=6\;\; \text{ and
}\;\;(\alpha|\beta)=-3.
\end{equation}
Hence, $\alpha$ is short root, while $\beta$ is long root.

We set \begin{equation} \label{nab} n(\gamma',\gamma):=\langle
\gamma',\gamma^\vee\rangle=
\frac{2(\gamma'|\gamma)}{(\gamma|\gamma)},\end{equation} (see
\cite [Chap.~VI, \S 1.1~(7)]{Bou}).
We will denote by
$s_\gamma$ the reflection in $W$ which corresponds to
$\gamma$, \ie
$s_\gamma(x):=x-\langle x,\gamma^\vee\rangle\gamma$.
We set $a:=s_\alpha$, $b:=s_\beta$ and $r:=ba$.

The Cartan matrix for $\rG_2(\FF)$ is $\left(\begin{matrix}
2& -1\\
-3& 2
\end{matrix}\right)$, and the values of $a$ and $b$ on the elements of
$\Phi^+$ are given in the table~1.

\begin{table}
\begin{center}
\begin{tabular}{|l|r|} \hline
$a(\alpha)=-\alpha$ & $a(\beta)=3\alpha+\beta$ \\ \hline
$a(\alpha+\beta)=2\alpha+\beta$& $a(2\alpha+\beta)=\alpha+\beta$
\\ \hline $\alpha(3\alpha+\beta)=\beta$&
$a(3\alpha+2\beta)=3\alpha+2\beta$\\ \hline
$b(\alpha)=\alpha+\beta$ & $b(\beta)=-\beta$ \\ \hline
$b(\alpha+\beta)=\alpha$ & $b(2\alpha+\beta)=2\alpha+\beta$ \\
\hline $b(3\alpha+\beta)=3\alpha+2\beta$&
$b(3\alpha+2\beta)=3\alpha+\beta$ \\ \hline
\end{tabular}
\caption{\label{tab:one} Values of $a$ and $b$.}
\end{center}
\end{table}

We write $B=TU$ for the corresponding Borel subgroup in $\rG_2(\FF)$ and
$\bar B=T \bar U$ for the opposite Borel subgroup. Denote by
$X(T)$ the group of rational characters of $T$. We have
\begin{equation} \label{XT}
X(T)=\Zset(2\alpha+\beta)+\Zset(\alpha+\beta).
\end{equation}
We identify $T \cong \FF^{\times} \times \FF^{\times}$ by
\begin{equation} \label{etaa}
\xi_\alpha\colon t \longmapsto ( (2\alpha+\beta)(t), (\alpha+\beta)(t) ).
\end{equation}
In this realization we have
$$
\begin{cases} \alpha (t_{1}, t_{2})= t_{1}t_{2}^{-1},\ \ \beta
(t_{1}, t_{2})=t_{1}^{-1}
t_{2}^{2} \\
a(t_{1}, t_{2})=(t_{2}, t_{1}),\ \ b(t_{1}, t_{2}) =(t_{1},
t_{1}t_{2}^{-1})
\end{cases}.
$$

The Weyl group $W=\Nor_{\rG_2(\FF)}(T)/T$ has $12$ elements. They are described along
with the action on the character  $\chi_1\otimes\chi_2$  of $T
\cong \FF^{\times} \times \FF^{\times}$:
\begin{equation} \label{elements}
\begin{matrix}
&  w & w(\chi_1\otimes\chi_2)\\
1.&e&\chi_1\otimes\chi_2\\
2.&a&\chi_2\otimes\chi_1\\
3.&b&\chi_1\chi_2\otimes\chi^{-1}_2 \\
4.&ab&\chi^{-1}_2\otimes \chi_1\chi_2\\
5.&ba&\chi_1\chi_2\otimes\chi^{-1}_1\\
6.&aba&\chi^{-1}_1 \otimes \chi_1\chi_2\\
7.&bab&\chi_1\otimes \chi^{-1}_1\chi^{-1}_2 \\
8.&abab&\chi^{-1}_1\chi^{-1}_2\otimes \chi_1\\
9.&baba&\chi_2 \otimes \chi^{-1}_1\chi^{-1}_2\\
10.&ababa&\chi^{-1}_1\chi^{-1}_2 \otimes\chi_2 \\
11.&babab&\chi^{-1}_2\otimes \chi^{-1}_1 \\
12.&bababa&\chi^{-1}_1\otimes \chi^{-1}_2
\end{matrix}.
\end{equation}
\subsection{Affine Weyl group, two-sided cells and unipotent orbits}
Let $\aW:=W\ltimes X(T^\vee)$ denote the affine Weyl group of the $p$-adic group
$G=\rG_2(F)$. Denote by $\{a,b,d\}$ the set of simple reflections in $\aW$, with
$W=\langle a,b\rangle$ and
$(ab)^6=(da)^2=(db)^3=\id$.

As in the case of an arbitrary Coxeter group, the group $\aW$ is
partitioned into \emph{two-sided cells}. The
definition of cells yields a natural partial ordering on the set
$\Cell(\aW)$ of two-sided cells in $\aW$. The \emph{highest} cell $\bc_e$
in this ordering contains just the identity element of $\aW$. Lusztig defined
in \cite{LCells} an \emph{$a$-invariant} for each two-sided cell.
The $a$-invariant respects (inversely) the partial
ordering on $\Cell(\aW)$.

The group $\aW$ has five two-sided cells $\bc_\id$, $\bc_1$,
$\bc_2$, $\bc_3$ and $\bc_0$ (see for instance \cite[\S 11.1]{Xi}) and the
ordering occurs to be total: $\bc_0\le \bc_3\le \bc_2\le \bc_1\le \bc_e$, with
\[\bc_\id=\left\{w\in \aW \,:\, a(w)=0\right\}=\{\id\},\]
\[\bc_1=\left\{w\in \aW \,:\, a(w)=1\right\},\]
\[\bc_2=\left\{w\in \aW \,:\, a(w)=2\right\},\]
\[\bc_3=\left\{w\in \aW \,:\, a(w)=3\right\},\]
\[\bc_0=\left\{w\in \aW \,:\, a(w)=6\right\} \;\;\text{(the \emph{lowest}
two-sided cell)}.\]
For a visual realization of the two-sided cells see \cite[p.~529]{Gu}.

Let $\cU$ denote the unipotent variety in the Langlands dual
$G^\vee=\rG_2(\Cset)$ of $G$. For $\orb$, $\orb'$ two unipotent classes in
$G^\vee$, we will write $\orb'\le\orb$ if $\orb'$ is contained in the Zariski
closure of $\orb$. The relation $\le$ defines a partial ordering on $\cU$.
In the Bala-Carter classification, the unipotent classes in $G^\vee$ are
$1\le \rA_1\le\widetilde \rA_1\le \rG_2(a_1)\le \rG_2$ (see for instance
\cite[p.~439]{Carter}).
The dimensions of these varieties are $0,\,6,\,8,\,10,\,12$. The
component groups are trivial except for $\rG_2(a_1)$ in which
case the component group is the symmetric group $S_3$.
The group $S_3$ admits $3$ irreducible
representations; two of these, the trivial and the $2$-dimensional
representations, namely $\rho_1$, $\rho_2$, appear in
our construction.
In \cite{Ram}, Ram refers to $1$, $\rA_1$,
$\widetilde \rA_1$, $\rG_2(a_1)$ and $\rG_2$ as the \emph{trivial},
\emph{minimal}, \emph{subminimal}, \emph{subregular} and \emph{regular}
orbit, respectively.

The bijection between $\Cell(\aW)$ and  $\cU$ that Lusztig has constructed
in \cite{LCellsIV} is order-preserving. Under this bijection,
$\bc_e$ corresponds to the \emph{regular} unipotent class and $\bc_0$
 corresponds to the \emph{trivial} class. If the
two-sided cell $\bc$ corresponds to the orbit of some unipotent element
$u\in G^\vee$, then $a(\bc)=\dim\cB_u$, where $\cB_u$ denotes the Springer
fibre of $u$ (that is, the set of Borel subgroups in $G^\vee$ containing
$u$). Lusztig's bijection is described as follows:
\[\bc_\id\leftrightarrow \rG_2\;\;\;\;\;
\bc_1\leftrightarrow \rG_2(a_1)\;\;\;\;\;
\bc_2\leftrightarrow\widetilde \rA_1\;\;\;\;\;
\bc_3\leftrightarrow \rA_1\;\;\;\;\;
\bc_0\leftrightarrow 1.\]

\subsection{Representations} \label{representations}
Let $R(\rG_2(F))$ denote the Grothendieck
group of admissible representations of finite length of $\rG_2(F)$.
With $\lambda \in \Psi(T)$, we will write $I(\lambda): = i^G_{T\subset B}(\lambda)$
for the induced representation (normalized induction).  We will
denote by $\ell(i_{T\subset B}^G(\lambda))$ for the length of this 
representation,
by $|i_{T\subset B}^G(\lambda)|$ the number of \emph{inequivalent} constituents.

It is a result of Keys that the unitary principal series $I(\chi_1\otimes
\chi_2 )$ is reducible if and only if $\chi_1$ and $\chi_2$ are different
characters of order $2$ \cite[Theorem~$\rG_2$]{K}. When reducible, it is of 
multiplicity one and lenght two. 

Let $\nu$ denote the normalized absolute value of $F$.
Using \cite[Prop. 3.1]{M} we have the following result: $I(\psi_1\chi_1
\otimes \psi_2 \chi_2)$, with $\psi_i$, $\chi_i$ ($i=1,2$) smooth
characters of $F^\times$, $\psi_1$, $\psi_2$ unramified, and
$\psi_1\chi_1$, $\psi_2\chi_2$ nonunitary, is reducible if and only if at
least one of the following holds:
\begin{equation}
\label{twelve}
\begin{array}{ccccc}
\psi_1\chi_1 =\nu&
\psi^2_1\psi_2\chi_1^2\chi_2=\nu&
\psi^2_1\psi_2\chi_1^2\chi_2=\nu^{-1}&
\psi_1\psi_2\chi_1\chi_2=\nu^{-1}
\cr
\psi_2\chi_2=\nu&
\psi_1\chi_1 =\nu^{-1}&
\psi_1\psi^2_2\chi_1\chi_2^2=\nu&
\psi_1\psi^2_2\chi_1\chi_2^2=\nu^{-1}\cr
\psi_2\chi_2 =\nu^{-1}&
\psi_1\psi_2\chi_1\chi_2=\nu&
\psi_1\psi^{-1}_2\chi_1\chi_2^{-1}=\nu&\psi_1\psi^{-1}_2\chi_1\chi_2^{-1}
=\nu^{-1},\end{array}\end{equation}
that is, if and only if there exists a root $\gamma\in\Phi$ such that
\begin{equation} \label{redcondition}
(\chi_1\otimes\chi_2)\circ\gamma^\vee=\nu^{\pm 1}.
\end{equation}

From now on we will assume that $\FF=F$, a local non Archimedean field.
Let $G=\rG_2$ and let $\fs=[T,\chi_1\otimes\chi_2]_G$.
Let $\Psi(F^\times)$ denote the group of unramified quasicharacters
of $F^\times$.
We have
\[D^\fs=\left\{\psi_1\chi_1\otimes\psi_2\chi_2\,:\,\psi_1,\psi_2\in
\Psi(F^\times)\right\}\cong
\left\{(z_1,z_2)\,:\,z_1,z_2\in\Cset^\times\right\}\,\cong\,T^\vee,\]
the Langlands dual of $T$, a complex torus of dimension $2$.

Let $\Psi^\temp(F^\times)$ denote the group of unramified unitary
quasicharacters of $F^\times$ and let $E=E^\fs$ be the maximal compact
subgroup of $D^\fs$. We have
\[E^{\fs}=\left\{\psi_1\chi_1 \otimes \psi_2\chi_2: \psi_1, \psi_2
\in \Psi^\temp(F^{\times})\right\}.\]

Let $w \in W(T)=W$. Then we have
\begin{equation} \label{inertie}
\fs=[T,\chi_1\otimes\chi_2]_G =
[T,w\cdot(\chi_1\otimes\chi_2)]_G.\end{equation}
We are also free to give $\chi$ an unramified twist:
this will not affect the inertial support.

\section{The Iwahori point in $\fT(\rG_2)$}
\label{Iwahoricase}

We will assume in this section that $\fs=\fii=[T,1]_G$.
We have $W^\fii=W$ and $W^\fii_\aff=\tW_\aff$.
The group $W$ is a finite Coxeter group of order $12$:
\[ W = \langle a,b, a^2 = b^2 = (ab)^6 = e \rangle.
\] Let $r = ab$. Then representatives of $W$-conjugacy classes are:
\[\{e,r,r^2,r^3,a,b\}.\]

\smallskip

\par
\noindent
{\bf Definition.}
We define the following partition of $T^\vee\q W$:
\begin{equation} \label{defn:extIwahori}
T^\vee\q W=\bigsqcup_{\bc_i\in\Cell(W_\aff)}(T^\vee\q
W)_{\bc_i},\end{equation}
where
\begin{eqnarray*}
(T^\vee\q W)_{\bc_e}&:=&(T^\vee)^r/\Cent(r),
\\
(T^\vee\q W)_{\bc_1}&:=&(T^\vee)^{r^3}/\Cent(r^3)\sqcup
(T^\vee)^{r^2}/\Cent(r^2),\\
(T^\vee\q W)_{\bc_2}&:=& (T^\vee)^{a}/\Cent(a),\\
(T^\vee\q W)_{\bc_3}&:=&(T^\vee)^{b}/\Cent(b),\\
(T^\vee\q W)_{\bc_0}&:=&T^\vee/W.
\end{eqnarray*}

\smallskip
{\sc Note 1.}
In the Definition above we had the freedom to permute $a$ and $b$ (that
is, we could as well have attached $(b)$ to $\bc_2$ and then $(a)$ to
$\bc_3$). 

\smallskip
\smallskip
{\sc Note 2.} The Springer correspondence
for the group $\rG_2$ (see \cite[p.~427]{Carter}) is as follows:
\[\begin{matrix}
\phi_{1,0}&\leftrightarrow&\bc_e\cr
\phi_{2,1}&\leftrightarrow&(\bc_1,\rho_1)\cr
\phi_{1,3}'&\leftrightarrow&(\bc_1,\rho_2)\cr
\phi_{2,2}&\leftrightarrow&\bc_2\cr
\phi_{1,3}''&\leftrightarrow&\bc_3\cr
\phi_{1,6}&\leftrightarrow&\bc_0.
\end{matrix}\]
Each of the following two bijections:
\[
\begin{matrix}
(r)&\leftrightarrow&\phi_{1,0}\cr
(r^3)&\leftrightarrow&\phi_{2,1}\cr
(r^2)&\leftrightarrow&\phi_{1,3}'\cr
(a)&\leftrightarrow&\phi_{2,2}\cr
(b)&\leftrightarrow&\phi_{1,3}''\cr
(e)&\leftrightarrow&\phi_{1,6}
\end{matrix}
\quad
\quad
\quad
\quad
\quad
\begin{matrix}
(r)&\leftrightarrow&\phi_{1,0}\cr
(r^3)&\leftrightarrow&\phi_{1,3}'\cr
(r^2)&\leftrightarrow&\phi_{2,1}\cr
(a)&\leftrightarrow&\phi_{2,2}\cr
(b)&\leftrightarrow&\phi_{1,3}''\cr
(e)&\leftrightarrow&\phi_{1,6}
\end{matrix}\]
by composing with the Springer correspondence, sends
\begin{itemize}
\item $(r)$ to the unipotent class in $G^{\vee}$ corresponding to
$\bc_e$, 
\item $(r^2)$ and $(r^3)$ to the unipotent class in
$G^{\vee}$ corresponding to $\bc_1$, 
\item $(a)$ to the unipotent
class in $G^{\vee}$ corresponding to $\bc_2$, 
\item$(b)$ to the
unipotent class in $G^{\vee}$ corresponding to $\bc_3$,
 \item $(e)$ to the unipotent
class in $G^{\vee}$ corresponding to $\bc_0$.
\end{itemize}

This is compatible with~(\ref{defn:extIwahori}).  Thanks to Jim
Humphreys for a helpful comment at this point.

\smallskip
\smallskip
{\sc Note~3.}
The correspondence between conjugacy classes in $W$ and  
unipotent classes in $G^\vee$ in Note~2 can be interpreted as 
a \emph{partition} of the set $\underW$ of conjugacy classes in $W$ indexed
by unipotent classes in $G^\vee$:
\begin{equation} \label{eqn:parti}
\underW=\bigsqcup_{u}\bbs_u,\end{equation}
where $u$ runs over the set of unipotent classes in $G^\vee$. Moreover, 
each $\bbs_u$ is a union of $n_u$ conjugacy classes in $W$, where 
$n_u$ is the number of isomorphism classes of irreducible representations of 
the component group $\Cent_{G^\vee}(u)/\Cent_{G^\vee}(u)^0$ which
appear in the Springer correspondence for $G^\vee$. In \cite[\S
8-9]{Lflagmanifold}, Lusztig defined similar kind of partitions in a more
general setting and more canonical way for adjoint algebraic reductive 
groups over an algebraic closure
of a finite field. The partition~(\ref{eqn:parti}) coincides with those
obtained by Lusztig for a group of type $\rG_2$ on the top of
page~7 of \cite{Lflagmanifold}.

\bigskip   
\smallskip

Let $J=J^\fii$ be the asymptotic Iwahori-Hecke algebra of $W$, let
$\mathbb{A}^1$ denote the affine complex line, let $\Iset$ denote the unit
interval, and let $\asymp$ be the geometrical equivalence defined in \cite[\S
4]{ABP}.

\begin{lem} \label{lem:extIwahori}
We have the following isomorphisms of algebraic varieties:
\begin{eqnarray*}
(T^\vee\q W)_{\bc_e}\iso pt_*\quad&\quad
(T^\vee\q W)_{\bc_1}\iso pt_1 \sqcup pt_2 \sqcup pt_3 \sqcup pt_4\\
(T^\vee\q W)_{\bc_2}\iso\mathbb{A}^1\quad&\quad
(T^\vee\q W)_{\bc_3}\iso\mathbb{A}^1,\end{eqnarray*}
\[E\q W \iso pt_* \sqcup (pt_1 \sqcup pt_2 \sqcup pt_3 \sqcup pt_4) \sqcup \mathbb{I} \sqcup \mathbb{I} \sqcup
E/W,\]
and $J\asymp \mathcal{O}(T^{\vee}\q W)$, where
$J_{\bc_1}\asymp\cO((T^\vee\q W)_{\bc_1})$, and
\begin{eqnarray*}
J_{\bc_e}\sim_{morita}\cO((T^\vee\q W)_{\bc_e})\quad&\quad
J_{\bc_2}\sim_{morita}\cO((T^\vee\q
W)_{\bc_2})\\
J_{\bc_3}\sim_{morita} \cO((T^\vee\q W)_{\bc_3})\quad&\quad
J_{\bc_0}\sim_{morita}\cO((T^\vee\q W)_{\bc_0}).
\end{eqnarray*}
\end{lem}
\begin{proof}
The centralizers in $W$ are:
\[
\Cent(e) = W,\;\;\; \Cent(r) = \langle r \rangle,\;\;\; \Cent(r^2)
= \langle r \rangle, \;\;\;\Cent(r^3) = W\] \[ \Cent(a) = \langle
r^3,a \rangle, \;\;\;\Cent(b) = \langle r^3,b \rangle \]

Case-by case analysis.  We will write $X = T^{\vee}$.
\begin{itemize}
\item $\bc = \bc_0$, $c= 1$. $X^{c}/\Cent(c) =
X/W$.
\item $\bc = \bc_e,\quad c = r$. $X^{c} = (1,1)$,
$X^{c}/\Cent(c) = pt_*$.
\item $\bc = \bc_1$,
$c = r^3$. $X^{c} =
\{(1,1),(-1,1),(1,-1),(-1,-1)\}$. Note that $(1,1)$ is fixed by
$W$ and \[(-1,-1) \sim_{bab} (-1,1) \sim_a (1,-1)\] are in a
single $W$-orbit. $X^{c}/\Cent(c) = pt_1 \sqcup pt_2$.
Also, attached to this cell, $c = r^2$. $X^{c} =
\{(1,1),(j,j),(j^2,j^2)\}$, where $j=\exp(2\pi i/3)$. Now \[(j,j) \sim_{ba}
(j^2,j^2)\] are in the same $\Cent(c)$-orbit.
$X^{c}/\Cent(c) = pt_3 \sqcup pt_4$.
\item
$\bc = \bc_2$,
$c = a$. $X^{c} = \{(z,z): z \in
\Cset^{\times}\}$.
\[X^{c}/\Cent(c) = \{\{(z,z),(z^{-1},z^{-1})\}: z \in
\Cset^{\times}\} \cong \mathbb{A}^1.\]
\item
$\bc = \bc_3$,
$c= b$. $X^{c} = \{(z,1): z \in \Cset^{\times}\}$.
\[X^{c}/\Cent(c) = \{\{(z,1),(z^{-1},1)\}: z \in
\Cset^{\times}\}\cong \mathbb{A}^1.\]
\end{itemize}

Let $F_\bc$ denote the maximal reductive subgroup of the centralizer in
$G^\vee$ of the unipotent class corresponding to $\bc$ and let
$R_{F_\bc}$ denote the complexified representation ring of $F_\bc$.

$\bullet$ We have $F_{\bc_0}=G^\vee$. Since the group $\rG_2$ is
$F$-split adjoint, we have (see \cite[Theorem~4]{ABP}):
\[J_{\bc_0} \simeq M_{|W|}(\cO(T^{\vee}/W)).\]
This isomorphism induces a homeomorphism of primitive ideal spectra:
\[\eta_{\bc_0}\colon T^\vee/W\to \Irr(J_{\bc_0}).
\]

$\bullet$ The reductive group $F_{\bc_{\id}}$ is the centre of $G^\vee$ and
$J_{\bc_\id} = \Cset$. Let $\eta_{\bc_\id}(pt_*)$ be the simple
module of $J_{\bc_\id}$.

$\bullet$ Let $i\in\{2,3\}$. We have $F_{\bc_i}\simeq\SL(2,\Cset)$.
In proving the Lusztig conjecture \cite[Conjecture~10.5]{LCellsIV} on the structure of the
asymptotic Hecke algebra, Xi constructed in \cite[\S 11.2]{Xi} an
isomorphism \[J_{\bc_i}\simeq\Mat_6(R_{F_{\bc_i}}).\]
Let $ \Cset[t,t^{-1}]$ denote the algebra of Laurent polynomials in one indeterminate $t$.
Let $\alpha$ denote the generator of $\Zset/2\Zset$. The group $\Zset/2\Zset$ acts as automorphisms of $\Cset[t,t^{-1}]$, with 
$\alpha(t) = t^{-1}$. The fixed algebra is isomorphic to the coordinate algebra of the affine line $\mathbb{A}^1$:
\[
 \Cset[t,t^{-1}]^{\Zset/2\Zset} \simeq \cO(\mathbb{A}^1), \quad \quad t + t^{-1} \mapsto t\]
 We then have
\[
R_{\SL(2,\Cset)} \simeq \Cset[t,t^{-1}]^{\Zset/2\Zset} \simeq \cO(\mathbb{A}^1)\]
The isomorphism
\[ J_{\bc_i} \simeq M_6(\cO(\mathbb{A}^1))\]
 induces a homeomorphism of primitive ideal spectra:
\[
\mathbb{A}^1 \simeq
\Irr(J_{\bc_i}).\]
This in turn determines a homeomorphism
\begin{equation} \label{eqn:JIwahorideuxtrois}
\eta_{\bc_i}: 
(T^\vee\q W)_{\bc_i}\simeq \Irr(J_{\bc_i})
\end{equation}
for $i = 2,3$.

$\bullet$ According to \cite[\S 11.2]{Xi}, we have $F_{\bc_1}=S_3$,
the symmetric group on $\{1,2,3\}$. We write $S_3
=\{e,(12),(13),(123),(132),(2,3)\}$.
Let \[\overline e=\{e\},\quad
\overline{(12)}=\{(12),(13),(23)\}\quad\text{and}\quad
\overline{(123)}=\{(123),(132)\}\] denote the three
conjugacy classes in $S_3$.
According to \cite[\S 11.3, \S 12]{Xi},
the based ring $J_{\bc_1}$ has four
simple modules: $(E_1, \pi_1)$, $(E_2, \pi_2)$, $(E_3, \pi_3)$, $(E_4,\pi_4)$
with $\dim E_1 = \dim E_2 =3$, $\dim E_3 = 2$, $\dim E_4 = 1$,
where
\begin{equation} \label{eqn:Jmods}
E_1=E_{{\overline e},\rho_1},\quad E_2=E_{\overline{(12)}},\quad
E_3=E_{\overline{(123)}}, \quad E_4=E_{\overline{e},\rho_2},\end{equation}
using the notation
\cite[\S 5.4]{Xi}.

Consider the map
$\delta_{\bc_1}\colon J_{\bc_1} \longrightarrow \Mat_3(\Cset) \oplus
\Mat_3(\Cset) \oplus
\Mat_2(\Cset) \oplus \Cset$, defined by
\[\delta_{\bc_1}(x)=(\pi_1(x),\pi_2(x),\pi_3(x),\pi_4(x)),\;\;
\text{for $x\in J_{\bc_1}$}.\] The map $\delta_{\bc_1}$
is spectrum-preserving. For the
primitive ideal space of $J_{\bc_1}$ is the discrete space $\{E_1,
E_2, E_3, E_4\}$ and the primitive ideal space of $\Mat_3(\Cset)
\oplus \Mat_3(\Cset) \oplus \Mat_2(\Cset) \oplus \Cset$ is $\{pt \sqcup
pt \sqcup pt \sqcup pt\}$.

Then we get $J_{\bc_1} \asymp \Cset^4\asymp\cO((T^\vee \q W)_{\bc_1})$.
Moreover we can choose the geometrical equivalence $J_{\bc_1} \asymp\cO((T^\vee \q W)_{\bc_1})$ in order that the induced bijection
$\eta_{\bc_1}\colon (T^\vee \q W)_{\bc_1}\to\Irr(J_{\bc_1})$ satisfies
\begin{equation} \label{eqn:etaIwahorione}
\eta_{\bc_1}(pt_i)=E_i\quad\text{for $1\le i\le 4$.}\end{equation}
\end{proof}

\begin{lem} \label{lem:familyIwahori} The flat family is given by
\[\mathfrak{X}_{\tau}: (1 - \tau^2 y)(x - \tau^2y) = 0.\]
\end{lem}
\begin{proof} 
Since there is only one quadratic unramified character, the unitary
principal series $I(\psi_1\otimes\psi_2)$ is always irreducible.

Then we write down all the nonunitary quasicharacters of $T$ which obey the
reducibility conditions~(\ref{twelve}):
\[\nu \otimes \psi_2,\; \nu^{-1} \otimes \psi_2,\;\psi_1 \otimes
\nu,\; \psi_1 \otimes \nu^{-1},\;
\psi\otimes\psi^{-2}\nu,\;\psi\otimes\psi^{-2}\nu^{-1},\;\]\[
\psi^{-2}\nu\otimes\psi,\; \psi^{-2}\nu^{-1}\otimes\psi,\;
\psi\otimes\psi^{-1}\nu,\; \psi\otimes\psi^{-1}\nu^{-1},\;
\psi\otimes\psi\nu^{-1},\; \psi\otimes\psi\nu\] with $\psi$ an
unramified quasicharacter of $F^{\times}$. These quasicharacters
of $T$ fall into two $W$-orbits:
\[
\begin{cases}
1.\ \ W\cdot(\nu\psi\otimes \psi),\\ 2.\ \ W\cdot(\nu \otimes \psi).
\end{cases}
\]
For the first $W$-orbit, we obtain the curve
\[\fC_2 = \{W\cdot (\nu\psi\otimes\psi)\,:\, \psi \in \Psi(F^{\times})\}
\cong \{W\cdot (z, q^{-1}z)\,: \, z \in \Cset^{\times}\}.\]
The second $W$-orbit creates the curve
\[
\fC_3 = \{W\cdot (\nu \otimes\psi)\,:\, \psi \in \psi(F^{\times}\}
\cong \{W\cdot (z, q^{-1})\,:\, z \in \Cset^{\times}\}.
\]
We obtain \[\fR=\fC_2\cup\fC_3.\] 

Let $\epsilon$ be the unique unramified quadratic character of $F^\times$, and let
$\omega$ denote an unramified cubic character of $F^\times$. In
the article of Ram~\cite{Ram} there is a list $t_a$, $\ldots$, $t_j$
of central characters, their calibration graphs, Langlands
parameters and indexing triples.  After computing the calibration
graphs, we are now able to identify these central characters with
points in the complex torus $T^{\vee} \cong \Psi(T)$:
\[t_a = (q^{-1}, q^{-2}) = \nu \otimes \nu^2,\; t_b =
(z, q^{-1}z) = \psi \otimes \nu\psi,\; t_c = (j, q^{-1}j) = \omega
\otimes \nu \omega,\]
\[t_d = (-1, q^{-1}) = \nu \otimes \epsilon
,\;t_e = (1, q^{-1}) = \nu \otimes 1,\; t_f = (q^{2/3},q^{-1/3}) =
\nu^{-2/3} \otimes \nu^{1/3},\]\[ t_g = (q^{1/2},q^{-1/2}) =
\nu^{-1/2} \otimes \nu^{1/2},\;t_h = (q^{-1}, z) = \nu \otimes
\psi,\]
\[t_i = (q^{-1},q^{-1}) = \nu
\otimes \nu,\;t_j = (q^{-1},q^{1/2}) = \nu\otimes\nu^{-1/2}\] with
$z = \psi(\varpi_F)$.   We have
\[\mathfrak{C}_2 \cap \mathfrak{C}_3 = \{t_a,t_d,t_e\}.\]

The flat family 
\[\mathfrak{X}_{\tau}: (1 - \tau^2 y)(x - \tau^2y) = 0\]
has the property that $\mathfrak{X}_{\sqrt q}=\fR$. 
\end{proof}

\begin{lem} \label{lem:cocharacterIwahori}
For each two-sided cell $\bc$ of $W^\fii$, the cocharacters
$h_\bc$ are as follows:
\[ h_{\bc}(\tau) = (1, \tau^{-2})  \textrm{ if } \bc = \bc_1, \bc_2, \bc_3\]   
\[h_{\bc_0} = 1,\; h_{\bc_e}(\tau) = (\tau^{-2},\tau^{-4})\] Now
define \[\pi_\tau(x) = \pi(h_{\bc}(\tau)\cdot x)\] for all $x$ in
the $\bc$-component.
 Then, for all $t \in T^{\vee}/W$ we have
\[
|\pi^{-1}_{\sqrt q}(t)| = |i_{T\subset B}^G(t)|.\]
\end{lem}

\begin{proof} 
We compare the description of the irreducible components of
$I(1\otimes\nu)$ given by Mui\'c~\cite{M} with those which occur
in Ram's table~\cite[p.20]{R}.
Then it follows from \cite[p. 476 and Prop.~4.3]{M} that
\[I(1\otimes\nu)=\pi(1)+\pi(1)'+J_\alpha(1/2,\delta(1))+2J_\beta(1/2,\delta(1))
+J_\beta(1,\pi(1,1))\]so that\[ \ell(I(1 \otimes \nu)) =
6,\quad\quad |I(1 \otimes \nu)| = 5.\]

When we collate the data in the table of Ram~\cite[p.20]{R}, we
find that
\[|i_{T\subset B}^G(t_e)| = 5\]
\[
|i_{T\subset B}^G(t)| = 4\quad \textrm{if}\quad t = t_a,\;t_c,\;t_d\]\[
|i_{T\subset B}^G(t)| = 2 \quad \textrm{if} \quad t = t_b, t_f, t_g, t_h,
t_i, t_j.\]

We will now compute a correcting system of cocharacters (see the paragraph after Theorem 1.5).   The extended quotient $T^{\vee} \q W$ is a disjoint union of $8$ irreducible
components $Z_1$, $Z_2$, $\ldots$, $Z_8$. Our notation is such that $Z_1 = pt_1$, $Z_2 = pt_2$, $Z_3 = pt_3$, $Z_4 = pt_4$, $Z_5 = pt_*$,
$Z_6 \simeq \mathbb{A}^1$, $Z_7 \simeq \mathbb{A}^1$, $Z_8 = T^{\vee}/W$.  

 In the following table, the first column comprises $8$ irreducible components 
$X_1, \ldots, X_8$ 
of $\widetilde{T^{\vee}}$ for which $\rho^{\fii}(X_j) = Z_j, \; j = 1, \ldots, 8$.  Let $[z_1,z_2]$ denote the image of $(z_1,z_2)$ via the standard 
quotient map $p^\fii\colon T^{\vee} \to T^{\vee}/W$, so that $[z_1,z_2] = W \cdot (z_1,z_2)$. 
When the first column itemizes the pairs
$(w,t) \in \widetilde{T^{\vee}}$, the second column
itemizes $p^\fii(h_j(\tau)t)$. The third column itemizes the corresponding
correcting cocharacters.
\[
\begin{array}{lll}
X_1 = (r^2,(j,j))  &  [j,\tau^{-2}j] & h_1(\tau) =  (1,\tau^{-2}) \\
X_2 = (r^2,(1,1))  &  [1,\tau^{-2}] & h_2(\tau) =  (1,\tau^{-2})   \\
X_3 = (r^3,(1,1))  &  [1,\tau^{-2}] & h_3(\tau) =  (1,\tau^{-2}) \\
X_4 = (r^3,(- 1,1)) &  [-1,\tau^{-2}] & h_4(\tau) =  (1,\tau^{-2}) \\
X_5 = (r,(1,1))  &  [\tau^{-2},\tau^{-4}] & h_5(\tau) =  (\tau^{-2}, \tau^{-4}) \\
X_6 =  \{(a,(z,z)): z \in \CC^{\times}\}  &  \{[z,\tau^{-2}z]: z \in \CC^{\times}\}  & h_6(\tau) =  (1,\tau^{-2}) \\ 
X_7  =  \{(b,(z,1)): z \in \CC^{\times}\}  &  \{[z,\tau^{-2}]: z \in \CC^{\times}\}  & h_7(\tau) =  (1,\tau^{-2}) \\
X_8 = \{(e,(z_1, z_2)): z_1, z_2  \in \CC^{\times}\} & \{[z_1,z_2]: z_1,z_2 \in \CC^{\times}\} & h_8(\tau) = 1 
\end{array}
\]

It is now clear that  cocharacters can assigned to two-sided cells as follows:
$h_{\bc}(\tau)  =  (1,\tau^{-2})$ if $\bc = \bc_1, \bc_2, \bc_3$, $h_{\bc_0} = 1,\; h_{\bc_e}(\tau) = (\tau^{-2}, \tau^{-4}).$

\smallskip

We are now in a position to write down explicitly the fibres $\pi_{\sqrt q}^{-1}(t)$ of the points of reducibility in the quotient
variety $T^{\vee}/W$.  We recall that $t_e = (1,q^{-1})$, $t_a=(q^{-1},q^{-2})$,
$t_d=(q^{-1},-1)$ 
and $t_c = (j, q^{-1}j)$.  We have, for example,
\[
\pi_{\sqrt q}^{-1}(t_e) = \{\rho^\fii(e,(1,q^{-1})), \rho^\fii(r^2,(1,1)), \rho^\fii(r^3, (1,1)), \rho^\fii(a,(1,1)), \rho^\fii(b,(1,1))\}
\]
\[
\pi_{\sqrt q}^{-1}(t_a) =
\{\rho^\fii(e,(q^{-1},q^{-2})),\rho^\fii(r,(1,1)),\rho^\fii(a,(q^{-1},q^{-1})),\rho^\fii(b,(q^{-1},1))\}\]
\[\pi_{\sqrt q}^{-1}(t_d) =\{\rho^\fii(e,(q^{-1},-1)),\rho^\fii(r^3,(-1,1)), 
\rho^\fii(a,(-1,-1)), \rho^\fii(b,(-1,1)),
\}\]
\[
\pi_{\sqrt q}^{-1}(t_c) = \{\rho^\fii(e,(j, j/q)), \rho^\fii(r^2,(j,j)), \rho^\fii(a,(j, j)), \rho^\fii(a,(j^2, j^2)) \}
\]

The two points $(a,(j,j))$ and $(a,(j^2,j^2))$ are especially interesting. The map $\pi_{\sqrt q}$ sends these two
points to the one point $t_c \in T^{\vee}/W$ since
$(j^2,j^2/q)$ and $(j,j/q)$ are in the same
$W$-orbit: $(j,j/q) \sim_{aba}(j^2,j^2/q)$. 
In the map 
\[
\pi_{\sqrt q} : T^{\vee} \q W \to T^{\vee}/W
\]
the image of the affine line $ (T^{\vee})^a / Z(a)$
has a \emph{self-intersection point}, namely $t_c$.

We have
\[|\pi^{-1}_{\sqrt q}(t_e)| = 5\]
\[|\pi^{-1}_{\sqrt q}(t)| = 4 \quad \textrm{if} \quad t = t_a,\;t_c,\;t_d.\]
\[|\pi^{-1}_{\sqrt q}(t)| = 2 \quad \textrm{if} \quad t = t_b, t_f, t_g, t_h,
t_i, t_j.\]
\end{proof}

\begin{lem} \label{lem:bijIwahori}  Part (4) of Theorem \ref{thm:main} is true for $\fii \in \mathfrak{B}(\rG_2)$.
\end{lem}
\begin{proof}
We will denote elements
in the five unipotent classes of $\rG_2(\Cset)$ by $1$, $u_3$, $u_2$, $u_1$,
$u_e$ (trivial,
minimal, subminimal, subregular, regular).

We recall that the irreducible $G$-module in $\Irr^\fii(G)$ corresponding to
the Kazhdan-Lusztig triple $(\sigma,u,\rho)$ is
denoted $\cV^\fii_{\sigma,u,\rho}$. According to
\cite[Table~6.1]{Ram}, we have
$\dim \cV_{t_e,u_1,\rho_2}^\fii=1$,
$\dim \cV_{t_e,u_1,\rho_1}^\fii=\dim \cV_{t_d,u_1,1}^\fii=3$,
and $\dim \cV_{t_c,u_1,1}^\fii=2$.
Hence we have
\[\cV_{t_e,u_1,\rho_2}^\fii=E_{4,q}, \quad\cV_{t_e,u_1,\rho_1}^\fii=E_{1,q},
\quad
\cV_{t_d,u_1,1}^\fii=E_{2,q},\quad \cV_{t_c,u_1,1}^\fii=E_{3,q},\]
where $E_{i,q}:=\phi_{q,\bc_1}^*(E_i)$, with $E_i$ as~(\ref{eqn:Jmods}).

The semisimple elements $t$, $\sigma$ below are always related as
in equation~(\ref{eqn:sigmat}), that is,
\[ \sigma =h_\bc(\sqrt q)\cdot t= \pi_{\sqrt q}(t).\]
Then the definition~(\ref{eqn:mu_c}) of $\mu^\fii$ gives (using the maps
$\eta_\bc$ defined in the proof of Lemma~\ref{lem:extIwahori}):
\[\mu^\fii(t) = \begin{cases}
\cV_{\sigma,1,1}^\fii\,,&\text{if $t\in T^{\vee}/W$,}\cr
\cV_{\sigma,u_2,1}^\fii\,,&\text{if $t \in \Aset^1$ (attached to $\bc_2$),}\cr
\cV_{\sigma,u_3,1}^\fii\,,&\text{if $t\in \Aset^1$ (attached to $\bc_3$);}
\end{cases}\]
two of the isolated
points are sent to the $L$-indistinguishable elements in the
discrete series which admit nonzero Iwahori fixed vectors:
\[\mu^\fii(pt_1) = \mathcal{V}_{t_e,u_1,\rho_1}^\fii,\quad \mu^\fii(pt_4) =
\mathcal{V}_{t_e,u_1,\rho_2}^\fii;\]
and
\[
\mu^\fii(pt_2) =  \mathcal{V}_{t_d, u_1,1}^\fii\]
\[\mu^\fii(pt_3) =  \mathcal{V}_{t_c,u_1,1}^\fii\]
\[\mu^\fii(pt_*) = \mathcal{V}_{t_a,u_e,1}^\fii \]
Now the infinitesimal character of $\mathcal{V}_{\sigma,u,\rho}^\fii$
is $\sigma$, therefore the map $\mu^\fii$ satisfies
\[
inf.ch. \circ \mu^\fii = \pi_{\sqrt q}.\]

The map $\mu^\fii$ is compatible with the cell-partitions
\[\mu^\fii((T^\vee\q W)_\bc)\subset\Irr^\fii(G)_\bc.\]
\end{proof}

\begin{lem} \label{lem:tempIwahori}  Part (5) of Theorem \ref{thm:main} is true for $\fii \in \mathfrak{B}(\rG_2)$.
\end{lem}
\begin{proof}
As for the compact extended quotient, this is accounted for as
follows: The compact quotient $E/W$ is the unitary principal
series, one unit interval is one intermediate unitary 
series, the other unit interval is the other intermediate unitary
series, and the five isolated points are the remaining tempered
representations itemized in \cite{Ram, M}.
\end{proof}

{\sc Note.} Among the tempered representations of $G$ which admit non-zero Iwahori
fixed vectors, those which have real central character are in bijection
with the conjugacy classes in $W$. For $G$ of type $\rG_2$, they are (see
\cite[Fig.~6.1, Tab.~6.3]{Ram}):
\[\cV^{\fii}_{t_0,1,1},\;\cV^{\fii}_{t_g,u_3,1},\;\cV^{\fii}_{t_j,u_2,1},\;
\cV^{\fii}_{t_e,u_1,\rho_1},\;\cV^{\fii}_{t_e,u_1,\rho_2},\;
\cV^{\fii}_{t_a,u_e,1}.\] 
These representations correspond (via the inverse map of $\mu^\fii$) to points in $T^\vee/W$, 
$(T^\vee\q W)_{\bc_3}$, $(T^\vee\q W)_{\bc_2}$, 
$pt_1\in (T^\vee\q W)_{\bc_1}$, $pt_4\in (T^\vee\q W)_{\bc_1}$, 
and $pt_*\in(T^\vee\q W)_{\bc_e}$,
respectively.

Hence the correspondence with conjugacy classes in $W$ that we obtained 
is the following:
\[\begin{matrix}
\cV^{\fii}_{t_0,1,1}&\leftrightarrow&(e)\cr
\\
\cV^{\fii}_{t_g,u_3,1}&\leftrightarrow&(b)\cr
\\
\cV^{\fii}_{t_j,u_2,1}&\leftrightarrow&(a)\cr
\\
\cV^{\fii}_{t_e,u_1,\rho_1}&\leftrightarrow&(r^3)\cr
\\
\cV^{\fii}_{t_e,u_1,\rho_2}&\leftrightarrow&(r^2)\cr
\\
\cV^{\fii}_{t_a,u_e,1}&\leftrightarrow&(r)\end{matrix}.\]
\section{Some preparatory results}
\subsection{The group $W^\fs$}
When $W^\fs=\{e\}$, our conjecture is easily verified.
\begin{lem} \label{lem: red}
We have $W^\fs\ne\{e\}$ if and only if $\fs=[T,\chi\otimes\chi]_G$ or
$\fs=[T,\chi\otimes 1]_G$ with $\chi$ an irreducible character of
$F^\times$.
\end{lem}
\begin{proof}
From~$(\ref{Ws})$, we have
\[W^\fs=\left\{w\in W\;:\;w\cdot(\chi_1\otimes\chi_2)=\psi(\chi_1\otimes\chi_2)
\;\text{ for some $\psi\in\Psi(T)$}\right\}.\]
Let $\sigma^\circ:=
\chi_1|_{\integers_F^\times}\otimes\chi_2|_{\integers_F^\times}$. Then we
get
\begin{equation} \label{DesWs}
W^\fs=\{w\in W\;:\;w\cdot\sigma^\circ=\sigma^\circ\}.
\end{equation}
Let $\chi_i^\circ:=\chi_i|_{\integers_F^\times}$.
From~(\ref{elements}), it follows that we have
$W^\fs=\{e\}$ if and only if
\begin{equation} \label{Wstrivial}
\chi_1^\circ\ne 1,\;\; \chi_2^\circ\ne 1,\;\;
\chi_1^\circ\chi_2^\circ\ne 1,\;\;\chi_1^\circ\ne\chi_2^\circ,\;\;
(\chi_1^\circ)^2\chi_2^\circ\ne 1,\;\;\chi_1^\circ(\chi_2^\circ)^2\ne 1.
\end{equation}
Hence we have $W^\fs\ne\{e\}$ if and only if we are in one of the
following cases:
\begin{enumerate}
\item We have $\chi_1^\circ=\chi_2^\circ$.
We may and do assume that $\chi_1=\chi_2=\chi$.
\item We have $\chi_2^\circ=1$. We may and do assume that
$\chi_1=\chi$ and $\chi_2=1$.
\end{enumerate}
\end{proof}
\begin{rem} \label{rem: WsI}
{\rm
We observe that the condition~(\ref{Wstrivial}) is equivalent to the condition
\[((\chi_1\otimes\chi_2)\circ\gamma^\vee)|_{\integers_F^\times}\ne 1,\;\;
\text{for all $\gamma\in\Phi$.}\]
Note that this condition is closely related to the
condition~(\ref{redcondition}).}
\end{rem}
\begin{rem} \label{rem: WsII}
{\rm
The group $W^\fs$ admits the following description
(which is compatible with~\cite[Lemma~6.2]{Roc}):
\[W^\fs=\langle s_\gamma\;:\;\text{$\gamma\in\Phi$ such that
$((\chi_1\otimes\chi_2)\circ\gamma^\vee)|_{\integers_F^\times}=1$}\rangle.
\]
In particular, this shows that $W^\fs$ is a finite Weyl group.}
\end{rem}

\subsection{The list of cases to be studied} \label{subsec: cases}
\subsubsection{$W$-orbits} \label{subsub: orbits}
1. The orbit $W\cdot(\chi\otimes\chi)$ consists of the
following characters:
\begin{equation} \label{firstorbit}
 \chi\otimes\chi, \ \,\chi^{-1}\otimes\chi^{-1},
\ \, \chi^2\otimes \chi^{-1}, \ \, \chi^{-1}\otimes\chi^2, \ \,
\chi\otimes\chi^{-2}, \ \, \chi^{-2}\otimes\chi.\end{equation}
It follows that
\[W\cdot(\chi\otimes\chi)=
\begin{cases}\chi\otimes\chi,\ \,\chi\otimes 1,
\ \,1\otimes\chi&\text{if $\chi$ is quadratic,}\cr
\chi\otimes\chi,\ \,
\chi^{-1}\otimes\chi^{-1}&\text{if $\chi$ is
cubic.}\end{cases}\]
We have
\begin{equation} \label{orbit}
|W\cdot(\chi\otimes\chi)|=\begin{cases}
1&\text{if $\chi$ is trivial,}\cr
3&\text{if $\chi$ is quadratic,}\cr
2&\text{if $\chi$ is cubic,}\cr
6&\text{otherwise.}
\end{cases}
\end{equation}

2. The orbit $W\cdot(\chi\otimes 1)$ consists of the following characters:
\begin{equation} \label{secondorbit}
\chi\otimes 1, \ \,1\otimes\chi,
\ \, \chi\otimes \chi^{-1}, \ \, \chi^{-1}\otimes 1, \ \,
1\otimes\chi^{-1}, \ \, \chi^{-1}\otimes\chi.\end{equation}
If $\chi$ is quadratic, then we have
\[W\cdot(\chi\otimes 1)=\left\{\chi\otimes\chi,\ \,\chi\otimes 1, \ \,1\otimes\chi.\right\}\]
We have
\begin{equation} \label{orbitnext}
|W\cdot(\chi\otimes 1)|=\begin{cases}
1&\text{if $\chi$ is trivial,}\cr
3&\text{if $\chi$ is quadratic,}\cr
6&\text{otherwise.}
\end{cases}
\end{equation}

\subsubsection{The cases}
From now on we will assume that $W^\fs\ne\{e\}$. Then the above discussion
leads to the following cases:
\begin{itemize}
\item[(1)]
$\fs=\fii=[T,1]_G$. Here $W^\fs=W$. Already studied in
section~\ref{Iwahoricase}.
\item[(2)]
$\fs=[T,\chi\otimes 1]_G$ with $\chi$ ramified non-quadratic, see section~\ref{sec:GL}.
\item[(3)]
$\fs=[T,\chi\otimes\chi]_G$ with $\chi$ ramified, neither quadratic nor
cubic, see section~\ref{sec:GL}:
\item[(4)]
$\fs=[T,\chi\otimes\chi]_G$ with $\chi$ ramified cubic, see
section~\ref{sec:PGL}.
\item[(5)]
$\fs=[T,\chi\otimes\chi]_G$ with $\chi$ ramified quadratic, see
section~\ref{sec:SO}.
\end{itemize}

\subsection{Lengths of the induced representations} \label{subsec: lengths}
We fix homomorphisms $x_\gamma\colon F\to G$ and
$\zeta_\gamma\colon\SL(2,F)\to G$  
$\gamma\in\Phi$ such that (see \cite[(6.1.3)~(b)]{BTI}):
\[x_\gamma(u)=\zeta_\gamma\left(\begin{matrix}
1&u\cr
0&1\end{matrix}\right),\;\;\;\;
x_{-\gamma}(u)=\zeta_{-\gamma}\left(\begin{matrix}
1&0\cr
u&1\end{matrix}\right)
\;\;\text{
and }\;\;
\gamma^\vee(t)=\zeta_\gamma\left(\begin{matrix}
t&0\cr
0&t^{-1}\end{matrix}\right).\]
For $\gamma\in\{\alpha,\beta\}$,
let $P_\gamma$ be the maximal standard parabolic subgroup of $G$
corresponding to $\gamma$, and
$M_\gamma$ be the centralizer of the image of $(\gamma')^\vee$ in $G$, where
$\gamma'$ is the unique positive root orthogonal to $\gamma$, that is,
\[\gamma'=\begin{cases}
3\alpha+2\beta&\text{ if $\gamma=\alpha$,}\cr
2\alpha+\beta&\text{ if $\gamma=\beta$.}
\end{cases}\]
Then $M_\gamma$ is a Levi factor for $P_\gamma$.

We extend $\zeta_\gamma\colon\SL(2,F)\to M_\gamma$ to an isomorphism
$\zeta_\gamma\colon\GL(2,F)\to M_\gamma$ by
\[\zeta_\gamma\left(\left(\begin{matrix}t&0\cr 0&t\end{matrix}\right)
\right):=\zeta_{\gamma'}\left(\left(\begin{matrix}t&0\cr
0&t^{-1}\end{matrix}\right)\right),
\;\;\;\;\text{for $t\in F^\times$.}\]
Then the restriction to $T$ of the inverse map of $\zeta_\gamma$ coincides
with the isomorphism
$\xi_\gamma\colon T{\overset\sim\to} F^\times\times F^\times$, where
$\xi_\alpha$ has been defined in~(\ref{etaa}), while
\[\xi_\beta\colon t\mapsto
((\alpha+\beta)(t),\alpha(t)).\]
For $\gamma\in\{\alpha,\beta\}$, and $\sigma$ a smooth irreducible
representation of $\GL(2)$, let $I_\gamma(\sigma)$ denote the
parabolically induced representation of $G$ 
\begin{equation} \label{Iab}
\Ind_{M_\gamma\subset P_\gamma}^G(\sigma\circ\zeta_\gamma^{-1}).
\end{equation}

Let $\delta$ be the Steinberg representation of $\GL(2)$ and let
$\delta(\chi)$ denote the twist of $\delta$ by the one
dimensional representation $\chi\circ\det$. Then $\delta(\chi)$
is the unique irreducible subrepresentation of
$\Ind_{B}^{\GL(2)}(\nu^{1/2}\chi\otimes\nu^{-1/2}\chi)$.
It is square integrable if $\chi$ is
unitary. The representation $\delta$ has torsion number $1$, and so
all the twists
$\{\delta(\chi): \chi \in \Psi(F^{\times})\}$ are distinct.

The inertial support of the representation $I_\gamma(\delta(\chi))$ is
$[T,(\chi\otimes\chi)\circ\xi_\alpha]_G$ if $\gamma=\alpha$ and
$[T,(\chi\otimes\chi)\circ\xi_\beta]_G=[T,(\chi\otimes 1)\circ\xi_\alpha]_G$ if
$\gamma=\beta$. We observe the following consequence (which will be crucial in
the sequel of the paper):

\begin{prop} \label{prop:inertialsupports} The representations
$I_\alpha(\delta(\chi))$ and $I_\beta(\delta(\chi))$
have same inertial support when $\chi^2=1$ and have distinct
inertial supports otherwise.
\end{prop}
\begin{proof}
It follows from the orbit computation done in \S~\ref{subsub: orbits}.
\end{proof}
 
\begin{lem} \label{lem: lengths} Let $\chi$, $\psi$ be two characters of
$F^\times$, with $\psi$ unramified and $\chi$ ramified.  Let $\epsilon$ be the unique unramified quadratic character of
$F^{\times}$, and let $\omega$ denote an unramified cubic character of $F^{\times}$. 
We set
\begin{eqnarray*}
\cP_2&=&\left\{(\nu^{\pm 1/2},\chi),(\nu^{\pm 1/2}\epsilon,\chi)
\,:\,\text{ $\chi$ is quadratic}\right\},\\
\cP_3&=&\left\{(\nu^{\pm 1/2},\chi),(\nu^{\pm 1/2}\omega,\chi),
(\nu^{\pm 1/2}\omega^2,\chi)\,:\, \text{ $\chi$ is cubic}\right\},\\
\cP&=&\cP_2\cup\cP_3.
\end{eqnarray*}
Then we have
\begin{eqnarray*}
\ell\left(I(\nu^{-1/2}\psi\chi\otimes\nu^{1/2}\psi\chi)\right)&=&
|I(\nu^{-1/2}\psi\chi\otimes\nu^{1/2}\psi\chi)|=
\begin{cases}
4&\text{if $(\psi,\chi)\in\cP$,}\cr
2&\text{otherwise,}\end{cases}\\
\ell\left(I(\nu^{-1/2}\psi\chi\otimes\nu)\right)&=&
|I(\nu^{-1/2}\psi\chi\otimes\nu|=
\begin{cases}
4&\text{if $(\psi,\chi)\in\cP_2$,}\cr
2&\text{otherwise}.\end{cases}\end{eqnarray*}
\end{lem}
\begin{proof}
In $R(M_\alpha)$, we have (see for instance \cite[Proposition~1.1(ii)]{M}):
\[\Ind_{T(U\cap M_\alpha)}^{M_\alpha}(\nu^{-1/2}\psi\chi\otimes\nu^{1/2}\psi\chi)
=\delta(\psi\chi)\oplus(\psi\chi\circ\det).\]
Similarly, in $R(M_\beta)$ (using now \cite[Proposition~1.1(iii)]{M}), we get:
\[\Ind_{T(U\cap M_\beta)}^{M_\beta}(\nu^{-1/2}\psi\chi\otimes\nu)
=\delta(\psi\chi)\oplus(\psi\chi\circ\det).\]
Then, by transitivity of parabolic induction, we obtain
\begin{eqnarray*}
I(\nu^{-1/2}\psi\chi\otimes\nu^{1/2}\psi\chi))&=&
I_\alpha(\delta(\psi\chi))+I_\alpha(\psi\chi\circ\det),\\
I(\nu^{-1/2}\psi\chi\otimes\nu)&=&I_\beta(\delta(\psi\chi))+
I_\beta(\psi\chi\circ\det).\end{eqnarray*}
Applying the involution $D_G$ defined in \cite{A}, it follows from
\cite[Th.~1.7]{A} that,
for $\gamma\in\{\alpha,\beta\}$, the induced representations
$I_\gamma(\delta(\psi\chi))$ and $I_\gamma(\psi\chi\circ\det)$
have the same length.

To describe the length of $I_\alpha(\delta(\psi\chi))$, we write
$\psi=\nu^s$, $s\in\Cset$. Now, in a different notation, we write
\[I_\alpha(\Reel(s), \delta(\nu^{\sqrt{-1}\Ima(s)}\chi))=
I_\alpha(\delta(\psi \chi)).\]
Then \cite[Theorem 3.1~(i)]{M} implies the following conclusion:

1. If $\chi$ is neither quadratic nor cubic then $I_\alpha(\delta(\psi\chi))$ is
irreducible. Hence
$\ell\left(I(\nu^{-1/2}\psi\chi\otimes\nu^{1/2}\psi\chi)\right)=2$.

2. If $\chi$ is ramified quadratic, then $I_\alpha(\delta(\psi
\chi))$ reduces if and only if $\nu^{\sqrt{-1}\Ima(s)}\in \{1,
\epsilon\}$ and $\Reel(s)=\pm 1/2$. Hence:
\begin{itemize}
\item
If $\psi\not\in\{\nu^{\pm 1/2},\nu^{\pm 1/2}\epsilon\}$,
then $\ell\left(I(\nu^{-1/2}\psi\chi\otimes\nu^{1/2}\psi\chi)\right)=2$.
\item
Otherwise, Rodier's result \cite[Corollary
on p.~419]{R} (see \cite[Prop. 4.1]{M}) implies that
$I(\nu^{-1/2}\psi\chi\otimes\nu^{1/2}\psi\chi)$ has length $4$ and
multiplicity $1$.
\end{itemize}
3. If $\chi$ is cubic ramified, then $I_\alpha(\delta(\psi\chi))$ reduces
if and only if $\nu^{\sqrt{-1}\Ima(s)}\in \{1,\omega,\omega^2\}$ and
$\Reel(s)=\pm 1/2$. Hence:
\begin{itemize}
\item If $\psi\not \in \{\nu^{\pm 1/2},\nu^{\pm 1/2}\omega,\nu^{\pm
1/2}\omega^2\}$, then $\ell\left(I(\nu^{-1/2}\psi\chi\otimes\nu^{1/2}\psi\chi)
\right)=2$.
\item Otherwise, it follows from {\it loc. cit.} that
$I(\nu^{-1/2}\psi\chi\otimes\nu^{1/2}\psi\chi)$ has length $4$ and
multiplicity $1$.
\end{itemize}

If $\chi$ is (ramified) not quadratic then $I_\beta(\delta(\psi\chi))$
is irreducible. Then
$\ell\left(I(\nu^{-1/2}\psi\chi\otimes\nu)\right)=2$.

We assume from now on that $\chi$ is quadratic ramified.
To describe the length of $I_\beta(\delta(\psi\chi))$, we write
$\psi=\nu^s\psi_0$, where $s\in\bR$, $\psi_0$
is unitary. Then $I_\beta(\delta(\psi\chi))$ reduces if and only
if $s=\pm 1/2$ and $\psi_0^2=1$. Therefore
the length of $I(\nu^{-1/2}\psi\chi\otimes\nu)$ is
two unless $\psi=\nu^{\pm 1/2}, \nu^{\pm 1/2}\epsilon$.
\end{proof}

\subsection{Two Lemmas}

The next two Lemmas will be needed in section~\ref{sec:SO} in the proof of
Lemma~\ref{lem:extSO}.

\subsubsection{Crossed product algebras}

Let $A$ be a unital $\Cset$-algebra and let $\Gamma$ be a finite group
acting as automorphisms of the unital $\Cset$-algebra $A$.
Let \[A^\Gamma:=\left\{a\in
A\,:\,\gamma\cdot a=a,\quad\forall \gamma\in\Gamma\right\}.\]
Let $A\rtimes\Gamma$ denote the crossed product algebra for the
action of $\Gamma$ on $A$: The elements of $A\rtimes\Gamma$ are formal sums
$\sum_{\gamma\in\Gamma}a_\gamma[\gamma]$, where:
\begin{itemize}
\item[$\bullet$]
the addition is
$(\sum_{\gamma \in \Gamma} a_{\gamma}[\gamma]) + (\sum_{\gamma
\in \Gamma}b_{\gamma}[\gamma]) = \sum_{\gamma \in
\Gamma}(a_{\gamma} + b_{\gamma})[\gamma]$,
\item[$\bullet$] the multiplication is determined by
$(a_{\gamma}[\gamma])(b_{\alpha}[\alpha]) =  a_{\gamma}(\gamma
\cdot b_{\alpha})[\gamma \alpha]$,
\item[$\bullet$]
the multiplication by $\lambda\in\Cset$ is
given by $\lambda(\sum_{\gamma \in \Gamma} a_{\gamma}[\gamma])=\sum_{\gamma \in
\Gamma} (\lambda a_{\gamma})[\gamma]$.
\end{itemize}
Let \[e_\Gamma:=|\Gamma|^{-1}\sum_{\gamma\in\Gamma}[\gamma].\]
Then $e_\Gamma$ is an idempotent (\ie $e_\Gamma^2=e_\Gamma$).
\begin{lem} \label{lem:crossedproduct}
The unital $\Cset$-algebras $A^\Gamma$ and
$(A\rtimes\Gamma)e_\Gamma(A\rtimes\Gamma)$ are Morita equivalent.
\end{lem}
\begin{proof}
See \cite[Proposition~p.25]{Ros}.
\end{proof}

\subsubsection{Ring homomorphisms}
\begin{lem} \label{lem:rings}
Let $A$ be a ring with unit and let $B$ be a ring (which is  not
required to have a unit). Let $\cJ\subset B$ be a two-sided ideal. Then any
surjective homomorphism of rings $\varphi\colon \cJ\twoheadrightarrow A$
extends \emph{uniquely} to a ring homomorphism $\tilde\varphi\colon B\to
A$.
\end{lem}
\begin{proof}
Choose $\theta_0\in \cJ$ such that $\varphi(\theta_0)=1_A$ (the unit in $A$).
Then, given $b\in B$, we define $\tilde\varphi(b)$ by
$\tilde\varphi(b):=\varphi(\theta_0 b)$.
\begin{enumerate}
\item
We will check first that $\tilde\varphi$ is well-defined, \ie that the definition
does not depend on the choice of $\theta_0$. Indeed, for every $\theta\in \cJ$
such that $\varphi(\theta)=1_A$,
then we have, on one hand:
\[\varphi(\theta b \theta_0)=\varphi(\theta
b)\varphi(\theta_0)=\varphi(\theta b),\]
and on the other hand:
\[\varphi(\theta b
\theta_0)=\varphi(\theta)\varphi(b\theta_0)=\varphi(b\theta_0).\]
Hence $\varphi(\theta b)=\varphi(b\theta_0)$. In particular,
we have $\varphi(\theta_0 b)=\varphi(b\theta_0)$. Thus $\varphi(\theta
b)=\varphi(\theta_0 b)$.
\item
Let $\tilde\varphi$ be any extension of $\varphi$. We have
\[\tilde\varphi(b)=1_A\tilde\varphi(b)=\varphi(\theta_0)\tilde\varphi(b)=
\tilde\varphi(\theta_0 b)=\varphi(\theta_0b),\]
since $\theta_0 b\in \cJ$.
\item
Finally, we check that $\tilde\varphi$ is a ring homomorphism. Indeed,
\[\tilde\varphi(b_1+b_2)=\varphi(\theta_0(b_1+b_2))=\varphi(\theta_0
b_1+\theta_0b_2)=\tilde\varphi(b_1)+\tilde\varphi(b_2);\]
\[\tilde\varphi(b_1b_2)=\varphi(\theta_0b_1b_2)=\varphi(\theta_0b_1b_2\theta_0)=\varphi(\theta_0b_1)\varphi(b_2\theta_0)=\tilde\varphi(b_1)\tilde\varphi(b_2).\]
\end{enumerate}
\end{proof}

\section{The two cases for which $H^\fs=\GL(2,\Cset)$} \label{sec:GL}

In this section, we will consider the following two cases.
\begin{enumerate}
\item[Case~1:]
We assume here that $\chi_2=1$ and $\chi_1=\chi$ with $\chi$
a ramified non-quadratic character. Then from~(\ref{secondorbit}) we
obtain
\[\begin{array}{ccccccc}\fs&=&[T,\chi\otimes 1]_G&=&[T,1\otimes\chi]_G&=&
[T,\chi\otimes \chi^{-1}]_G\cr
&=&[T,\chi^{-1}\otimes 1]_G&=&[T,1\otimes\chi^{-1}]_G&=&[T,\chi^{-1}\otimes\chi]_G.
\end{array}\]
It follows from~(\ref{elements}) that
\begin{equation} \label{eqn:Stwo}
W^\fs=\{e,b\}\cong S_2.\end{equation}
\item[Case~2:]
We assume that $\chi_1=\chi_2=\chi$ with $\chi$ a ramified
character which is neither quadratic nor cubic.
From~(\ref{firstorbit}) we obtain
\[\begin{array}{ccccccc}
\fs&=&[T,\chi\otimes\chi]_G&=&[T,\chi^{-1}\otimes\chi^{-1}]_G&=&
[T,\chi^2\otimes \chi^{-1}]_G\cr
&=&[T,\chi^{-1}\otimes\chi^2]_G&=&[T,\chi\otimes\chi^{-2}]_G
&=&[T,\chi^{-2}\otimes\chi]_G.
\end{array}\]
It follows from~(\ref{elements}) that
\begin{equation} \label{eqn:StwoII}
W^\fs=\{e,a\}\cong S_2.\end{equation}
\end{enumerate}
In both Case~1 and Case~2, we have $\tW_\aff^\fs=S_2\ltimes X(T^\vee)$.
Hence $\tW_\aff^\fs$ is the extended affine Weyl group of the
$p$-adic group $\GL(2,F)$. There are $2$ two-sided cells, say
$\bbb_e$ and $\bbb_0$, in $\tW_\aff^\fs$; they correspond to the
regular unipotent class $\cU_e$ and to the trivial unipotent class of
$\GL(2,\Cset)$, respectively. Hence $\bbb_e$ and $\bbb_0$
correspond to the partitions $(2)$ and $(1,1)$ of $2$,
respectively. We have $\bbb_0\le\bbb_e$.

\par

\medskip

\par
\noindent
{\bf Definition.}
We define the following partition of $T^\vee\q W^\fs$:
\begin{equation} \label{defn:extGL}
T^\vee\q W^\fs=(T^\vee\q W^\fs)_{\bbb_e}\sqcup(T^\vee\q
W^\fs)_{\bbb_0},\end{equation} where
$(T^\vee\q W^\fs)_{\bbb_e}:=(T^\vee)^c/\Cent(c)$, where $c$ is the
nontrivial element in $W^\fs$, and
$(T^\vee\q W^\fs)_{\bbb_0}:=T^\vee/W^\fs$.

\medskip

We will denote by $J^\fs$ the based ring of the extended affine
Weyl group $\tW_\aff^\fs$ defined in~(\ref{eqn:Wsa}) and
set
\begin{equation} \label{eq: U1} U(1): = \{z \in \Cset: |z| =
1\}.\end{equation}

\begin{lem} \label{lem:extGL}
We have the following diffeomorphisms
\[(T^\vee\q W^\fs)_{\bbb_e}\to\Cset^\times,\quad\quad E^\fs\q W^\fs\to U(1)\sqcup
E^\fs/W^\fs,\]
and we have
\[J^\fs=J_{\bbb_e}^\fs+J_{\bbb_0}^\fs\sim_{morita}
\mathcal{O}(T^{\vee}\q W^\fs),\] where
$J_{\bbb_e}^\fs\sim_{morita}\cO((T^\vee\q W^\fs)_{\bbb_e})$ and
$J_{\bbb_0}^\fs\sim_{morita}\cO((T^\vee\q W^\fs)_{\bbb_0})$.
\end{lem}

\begin{proof}
Let $D = D^\fs$ and $E:=E^\fs$. We give the case-by-case analysis.
\begin{itemize}
\item $c=1$.  $D^c/\Cent(c) = D/W^{\fs}$ and $E^c/\Cent(c) = E/W^{\fs}$.
\item $c\ne 1$.
\begin{enumerate}
\item[Case~1:] $c=b$:  \[D^b/\Cent(b)=D^b=\{(t,1):t\in\Cset^\times\}.
\quad
E^b=\{(t,1) : t \in U(1)\}.\]
\item[Case~2:] $c=a$:  \[D^a/\Cent(a)=D^a=\{(t,t):t\in\Cset^\times\}.
\quad
E^b=\{(t,t) : t \in U(1)\}.\]
\end{enumerate}
\end{itemize}
We have $J^\fs=J_{\bbb_e}+J_{\bbb_0}$ and (see \cite[proof of Theorem~3]{ABP}):
\[J_{\bbb_e}\sim_{morita}\cO(\Cset^\times),
\;\;\;
J_{\bbb_0}\sim_{morita}\cO((\Cset^\times)^2/S_2)\cong\cO(D^\fs/W^\fs).\]
It gives
\begin{equation} \label{eqn:MoritaGL}
J_{\bbb_i}\sim_{morita}\cO(T^\vee\q W^\fs)_{\bbb_i},\quad\text{for
$i\in\{e,0\}$.}\end{equation}
\end{proof}

\begin{lem} \label{lem:familyGL} The flat family is given by
\[
\mathfrak{X}_{\tau}: 1 - \tau y = 0, \quad \text{in Case~1;}\]
\[
\mathfrak{X}_{\tau}: x - \tau^2 y = 0, \quad \text{in Case~2.}\]\end{lem}
\begin{proof}
We will considerate the two cases separately.
\begin{enumerate}
\item[Case~1:]
The curve of reducibility $\fC_1=\fX_{\sqrt q}=\fR$ is given by
\[\fC_1 = \left\{\psi\chi\nu^{-1/2}\otimes\nu^{1/2}\,:\, \psi \in
\Psi(F^{\times})\right\} \cong \left\{(z\sqrt{q},1/\sqrt{q})\,:\,
z \in \Cset^{\times}\right\}\]
We write down all the {\bf nonunitary} quasicharacters $(\psi_1\chi\otimes\psi_2)$,
with $\psi_1,\psi_2\in\Psi(F^\times)$, which obey the reducibility
conditions~(\ref{twelve}):
\[\psi\chi\otimes\nu^{-1},\;\;\;\psi\chi\otimes\nu,\quad
\text{ with $\psi\in\Psi(F^\times)$.}\] We get only one
$W^\fs$-orbit of characters.
\item[Case~2:]
The curve of reducibility $\fC_2=\fX_{\sqrt q}=\fR$ is given by
\[
\mathfrak{C}_2 =
\left\{\psi\chi\nu^{-1/2}\otimes\psi\chi\nu^{1/2}\,:\, \psi \in
\Psi(F^{\times})\right\} \cong \left\{(z\sqrt{q},z/\sqrt{q})\,:\,
z \in \Cset^{\times}\right\}.\]
 We write down all the
{\bf nonunitary} quasicharacters of $T$ which obey the reducibility
conditions~(\ref{twelve}):
\[\psi\chi\otimes\psi\nu^{-1}\chi,\;\;\;\psi\chi\otimes\psi\nu\chi,\quad
\text{ with $\psi\in\Psi(F^\times)$.}\] We get only one
$W^\fs$-orbit of characters. Indeed,
\begin{itemize}
\item[$\bullet$] the family of characters
$\{\psi\chi\otimes\psi\nu\chi:\psi \in \Psi(F^{\times})\}$, with
the change of variable $\phi:=\psi\nu^{1/2}$ is
$\{\phi\nu^{-1/2}\chi\otimes\phi\nu^{1/2}\chi: \phi \in
\Psi(F^{\times})\}$;
\item[$\bullet$] the family of characters
$\{\psi\chi\otimes\psi\nu^{-1}\chi:\psi \in \Psi(F^{\times})\}$,
with the change of variable $\phi:=\psi\nu^{-1/2}$ is
$\{\phi\nu^{1/2}\chi\otimes\phi\nu^{-1/2}\chi: \phi \in
\Psi(F^{\times})\}$; by applying $a$, we then get
$\{\phi\nu^{-1/2}\chi\otimes\phi\nu^{1/2}\chi: \phi \in
\Psi(F^{\times})\}$.
\end{itemize}
\end{enumerate}
\end{proof}

\begin{lem} \label{lem:cocharacterGL}
The cocharacters are as follows:
\[h_{\bbb_0} = 1,\;\; h_{\bbb_e}(\tau) =
(\tau,\tau^{-1})\] which leads to \[\pi_\tau(v) =
\pi(h_{\bbb_i}(\tau)\cdot v)\] for all $v$ in the
$\bbb_i$-component, $i\in\{0,e\}$.
\end{lem}

\begin{proof} We apply Lemma~\ref{lem:extGL}.
For all $v\in D^\fs/W^\fs$ we have
\[|\pi^{-1}_{\sqrt q}(v)| = |i_{T\subset B}^G(v)|.\]

If $v\notin\fC\cup\fC'$, we have $|i_{T\subset B}^G(v)|=1=|\pi^{-1}_{\sqrt q}(v)|$. On
the other hand, for each $v\in\fC\cup\fC'$, from Lemma~\ref{lem: lengths} we have
\[\ell(i_{T\subset B}^G(v))=|i_{T\subset B}^G(v)|=2=|\pi_{\sqrt q}^{-1}(v)|,\]
due to Lemma~\ref{lem:familyGL}.
\end{proof}

\begin{lem} \label{lem:bijGL} Part (4) of Theorem \ref{thm:main} is true for the points
$\fs=[T,\chi\otimes 1]_G$, with $\chi$ ramified non quadratic and
$\fs=[T,\chi\otimes\chi]_G$, with $\chi$ ramified neither
quadratic nor cubic.
\end{lem}
\begin{proof}
The semisimple elements $v, \sigma$ are always related as follows
$\sigma = \pi_{\sqrt q}(v)$.
Let $\eta^\fs\colon(T^\vee//W^\fs)\to\Irr(J^\fs)$ be the bijection which
is induced by the Morita equivalences in~(\ref{eqn:MoritaGL}).
Then the definition~(\ref{eqn:mu_c}) of $\mu^\fs\colon
(T^\vee\q W^\fs)\to \Irr(G)^\fs$ gives:
\[\mu^\fs(v) =\begin{cases}  \mathcal{V}_{\sigma,1,1}^\fs\,,&\text{if  $v \in T^\vee/W^\fs$,}\cr
\mathcal{V}_{\sigma,u_e,1}^\fs\,,&\text{if $v \in (T^\vee\q W^\fs)_{\bbb_e}$.}
\end{cases}\]
Now the infinitesimal character of $\mathcal{V}_{\sigma,u,\rho}^\fs$
is $\sigma$, therefore the map $\mu^\fs$ satisfies
\[inf.ch. \circ \mu^\fs = \pi_{\sqrt q}.\]
\end{proof}

\begin{lem} \label{lem:tempGL}Part (5) of Theorem \ref{thm:main} is true. \end{lem}
\begin{proof}
As for the compact extended quotient, this is accounted for as
follows: The compact quotient $E/W$ is sent to the unitary principal
series
\[\{I(\psi_1\chi\otimes\psi_2\chi)\,:\,
\psi_1, \psi_2 \in \Psi(F^{\times})\}/W^{\fs}\]
and the other component $U(1)$ to the intermediate unitary 
series \[\{I_{\alpha}(\delta(\psi\chi)): \psi \in \Psi^{\temp}(F^{\times})\}.\]
\end{proof}

\section{The case $H^\fs=\SL(3,\Cset)$} \label{sec:PGL}
We assume in this section that $\chi_1=\chi_2=\chi$, with $\chi$ a
ramified character of order $3$.
We have
\[\fs=[T,\chi\otimes\chi]_G = [T,\chi^{-1}\otimes\chi^{-1}]_G=
[T,\chi\otimes\chi^{-1}]_G=[T,\chi^{-1}\otimes\chi]_G.\]
It follows from~(\ref{elements}) that
\begin{equation} \label{eqn:Stree}
W^\fs=\{e,a,bab,abab,baba,ababa\}\cong S_3.\end{equation}

We have $a=s_\alpha$ and $bab=s_{\alpha+\beta}$. We observe that the root
lattice $\Zset\alpha\oplus \Zset(\alpha+\beta)$ equals $X(T)$. It follows
that $\tW_\aff^\fs$ (as defined in~(\ref{eqn:Wsa})) is the extended affine Weyl group
of the $p$-adic group $G_\fs=\PGL(3,F)$.

There are $3$ two-sided cells $\bd_0$, $\bd_1$, $\bd_e$ in $\tW_\aff^\fs$,
they are in bijection the $3$ unipotent classes of $\SL(3,\Cset)$.
The two-sided cell $\bd_0$ corresponds to the trivial
unipotent class, $\bd_1$ corresponds to the subregular unipotent class,
and $\bd_e$ corresponds to the regular unipotent one. Hence
we have $\bd_0\le\bd_1\le\bd_e$ and $\bd_e$, $\bd_1$, $\bd_0$
correspond respectively to the partitions $(3)$, $(2,1)$ and
$(1,1,1)$ of $3$.
We will denote elements in the three unipotent classes by $1$, $u_1$, $u_e$
(trivial, subregular, regular).
The group $W^\fs$ admits $3$ conjugacy classes:
$\{e\}$, $\{a,bab,ababa\}$ and $\{abab,baba\}$.
We recall that $r=ba$.

\smallskip

\par
\noindent
{\bf Definition.}
We define the following partition of $T^\vee\q W^\fs$:
\begin{equation} \label{defn:extSL}
T^\vee\q W^\fs=(T^\vee\q W^\fs)_{\bd_e}\sqcup(T^\vee\q W^\fs)_{\bd_1}
\sqcup(T^\vee\q W^\fs)_{\bd_0},\end{equation}
where
\begin{eqnarray*}
(T^\vee\q W^\fs)_{\bd_e}&:=&(T^\vee)^{r^2}/\Cent(r^2),
\\
(T^\vee\q W^\fs)_{\bd_1}&:=&(T^\vee)^{a}/\Cent(a),\\
(T^\vee\q W^\fs)_{\bd_0}&:=&T^\vee/W^\fs.
\end{eqnarray*}

\begin{lem} \label{lem:extSL} We have
\begin{eqnarray*}
(T^\vee\q W^\fs)_{\bd_e}&=&pt_1 \sqcup pt_2 \sqcup pt_3
\\(T^\vee\q W^\fs)_{\bd_1}&=&\Cset^\times\\
E^\fs\q W^\fs &=& (pt_1 \sqcup pt_2 \sqcup pt_3) \sqcup U(1) \sqcup E^\fs/W^\fs,
\end{eqnarray*}
and $J^\fs\sim_{morita} \mathcal{O}(T^{\vee}\q W^\fs)$, where
\begin{eqnarray*}
J_{\bd_e}&\sim_{morita}&\cO((T^\vee\q W^\fs)_{\bd_e}),\\
J_{\bd_1}&\sim_{morita}&\cO((T^\vee\q W^\fs)_{\bd_1}),\\
J_{\bd_0}&\sim_{morita}&\cO((T^\vee\q W^\fs)_{\bd_0}).
\end{eqnarray*}
\end{lem}
\begin{proof}
We have
$\Cent_{W^\fs}(a)=\{e,a\}$ and $\Cent_{W^\fs}(abab)= \{e,abab,baba\}$.
Let $D:=D^\fs$ and $E:=E^\fs$. We obtain
\[D^{a}=\left\{(t,t)\,:\,t\in\Cset^\times\right\}\;\;\text{ and }\;\;
D^{abab}= \left\{(t,t^{-1})\,:\,t\in\Cset^\times\right\}.\] Case
by case analysis.
\begin{itemize}
\item[$\bullet$] $\bd=\bd_e$, $c=ab ab$. $X^c/\Cent_{W^\fs}(c)= \{(1,1),(j,j^2),(j^2,j)\}
=E^c/\Cent_{W^\fs}(c)$,
where $j$ is a primitive third root of unity. The points $(1,1)$,
$(j,j^2)$, $(j^2,j)$ belong to $3$ different $\Cent_{W^\fs}(c)$-orbits.
Therefore \[D^c/\Cent_{W^\fs}(c)=E^c/\Cent_{W^\fs}(c)=pt_1 \sqcup pt_2 \sqcup pt_3.\]
\item[$\bullet$] $\bd=\bd_1$, $c=a$. $D^c/\Cent_{W^\fs}(c)=D^c\cong\Cset^\times$.
\[E^c/\Cent_{W^\fs}(c)=E^c=\{(t,t):t\in U(1)\}.\]
\item[$\bullet$] $\bd=\bd_0$, $c=1$. $D^c/\Cent_{W^\fs}(c)=D/W^\fs$.
$E^c/\Cent_{W^\fs}(c)=E/W^\fs$.
\end{itemize}
We have $J^\fs=J_{\bd_e}+J_{\bd_1}+J_{\bd_0}$ and
(see \cite[proof of Theorem~4]{ABP}):
\[J_{\bd_e}\sim_{morita}\Cset^3,\;\;\;
J_{\bd_1}\sim_{morita}\cO(\Cset^\times),\;\;\;
J_{\bd_0}\sim_{morita}\cO(D^\fs/W^\fs).\]
It gives
\begin{equation} \label{eqn:MoritaSL}
J_{\bd_i}\sim_{morita}\cO((T^\vee\q W^\fs)_{\bd_i}),\;\;\text{for
$i\in\{e,1,0\}$.}
\end{equation}
\end{proof}

\begin{lem} \label{lem:familySL}  The flat family is given by
\[
\mathfrak{X}_{\tau}: (x - \tau^2 y = 0)\,\cup\,(j,\tau^{-2}j^2)\cup
(j^2,\tau^{-2}j).\]
\end{lem}

\begin{proof} 
We write down all the quasicharacters of $T$ which obey the
reducibility conditions~(\ref{twelve}):
\[\begin{cases}
1. \ \ \psi\chi\otimes\psi^{-2}\nu\chi\\
2. \ \ \psi\chi\otimes\psi^{-2}\nu^{-1}\chi\\
3. \ \ \psi^{-2}\nu\chi\otimes\psi\chi\\
4. \ \ \psi^{-2}\nu^{-1}\chi\otimes\psi\chi\\
5. \ \ \psi\chi\otimes\psi\nu^{-1}\chi\\
6. \ \ \psi\chi\otimes\psi\nu\chi\\
\end{cases}\quad\text{with $\psi \in \Psi(F^{\times})$.}\]
We get only one $W$-orbit of characters. Indeed,
\begin{itemize}
\item[6:] the family of characters
$\{\psi\chi\otimes\psi\nu\chi:\psi \in \Psi(F^{\times})\}$, with the change of
variable $\phi:=\psi\nu^{1/2}$ is
$\{\phi\nu^{-1/2}\chi\otimes\phi\nu^{1/2}\chi: \phi \in
\Psi(F^{\times})\}$;
\item[5:] the family of characters
$\{\psi\chi\otimes\psi\nu^{-1}\chi:\psi \in \Psi(F^{\times})\}$, with the
change of variable $\phi:=\psi\nu^{-1/2}$ is
$\{\phi\nu^{1/2}\chi\otimes\phi\nu^{-1/2}\chi: \phi \in
\Psi(F^{\times})\}$; by applying $a$, we then get
$\{\phi\nu^{-1/2}\chi\otimes\phi\nu^{1/2}\chi: \phi \in
\Psi(F^{\times})\}$;
\item[4:] we have
$baba(\psi^{-2}\nu^{-1}\chi\otimes\psi\chi)=
\psi\chi\otimes\psi\nu\chi$, which belongs to the family of
characters in 6;
\item[3:] we have
$baba(\psi^{-2}\nu\chi\otimes\psi\chi)=
\psi\chi\otimes\psi\nu^{-1}\chi$, which belongs to the family of
characters in 5;
\item[2:] we have
$bab(\psi\chi\otimes\psi^{-2}\nu^{-1}\chi)=
\psi\chi\otimes\psi\nu\chi$, which belongs to the family of
characters in 6;
\item[1:] we have $ba
b(\psi\chi\otimes\psi^{-2}\nu\chi)=\psi\chi\otimes\psi\nu^{-1}\chi$,
which belongs to the family of characters in 5.
\end{itemize}
The induced representations
$\left\{I(\phi\nu^{-1/2}\chi\otimes\phi\nu^{1/2}\chi)\,:\, \phi
\in \Psi(F^{\times})\right\}$ are pa\-ra\-me\-tri\-zed by
$\left\{(zq^{1/2},zq^{-1/2})\,:\,z\in\Cset^\times\right\}=\
\left\{(z,zq^{-1})\,:\,z\in\Cset^\times\right\}$.
\end{proof}

We set
\begin{gather}
y_a=(1,q^{-1}), \;\;\; y'_a=(j,q^{-1}j^2),
\;\;\;y_a''=(j^2,q^{-1}j),\notag\\
y_b=(z,q^{-1}z),\quad\text{where
$z\notin\{1,j,j^2\}$.}\notag
\end{gather}

\begin{lem} \label{lem:cocharacterSL}
Define, for each two-sided cell $\bd$ of $W^\fs_\aff$,  cocharacters
$h_\bd$ as follows:
\[h_{\bd_0}=1,\;\;\; h_{\bd_i}(\tau)=
(1,\tau^{-2})\;\;\;\text{for $i=1,e$.}\]
Then, for all $y \in D^\fs/W^\fs$ we have
\[|\pi^{-1}_{\sqrt q}(y)| = |i_{T\subset B}^G(y)|.\]
\end{lem}
\begin{proof}
If $y\notin\{y_a, y'_a,y_a'',y_b\}$, then we have
$|i_{T\subset B}^G(y)|=1=|\pi^{-1}_{\sqrt q}(y)|$. On the other hand, Lemma~\ref{lem: lengths} gives
\[|i_{T\subset B}^G(y_b)|=2=|\pi^{-1}_{\sqrt q}(y_b)|.\]
This leads to the following points of length $4$:
\[\chi\otimes\nu\chi,\quad \omega\chi\otimes\nu\omega\chi, \quad
\omega^2\chi\otimes\nu\omega^2\chi,\quad \nu^{-1}\chi\otimes
\chi,\quad \nu^{-1}\omega\chi\otimes\omega\chi,\quad
\nu^{-1}\omega^2\chi\otimes\omega^2\chi,\]
where $\omega$ denotes an unramified cubic character of $F^\times$. Since
\begin{eqnarray*}
baba(\nu^{-1}\chi\otimes\chi) &=& \chi\otimes\nu\chi,\\
baba(\nu^{-1}\omega\chi\otimes\omega\chi)
&=&\omega\chi\otimes\nu\omega\chi,\\
baba(\nu^{-1}\omega^2\chi\otimes\omega^2\chi)
&=&\omega^2\chi\otimes\nu\omega^2\chi,
\end{eqnarray*} this leads to exactly $3$
points in the Bernstein variety $\Omega^{\mathfrak{s}}(G)$ which
pa\-ra\-me\-tri\-ze representations of length $4$, namely
\[[T,\chi\otimes\nu\chi]_G,\;\;\;[T,\omega\chi\otimes\nu\omega\chi]_G,\;\;\;
[T,\omega^2\chi\otimes\nu\omega^2\chi]_G.\]
The coordinates of these points in the algebraic surface
$\Omega^{\mathfrak{s}}(G)$ are $y_a$, $y_a'$, $y_a''$, respectively.

We will now compute a correcting system of cocharacters.  
The extended quotient $T^{\vee} \q W^\fs$ is a disjoint
union of $5$ irreducible
components $Z_1$, $Z_2$, $\ldots$, $Z_5$. Our notation is such that $Z_1 = pt_1$,
$Z_2 = pt_2$, $Z_3 = pt_3$, $Z_4\simeq \mathbb{A}^1$, $Z_5 =
T^{\vee}/W^\fs$.

In the following table, the first column comprises $5$ irreducible
components $X_1$, $\ldots$, $X_5$ 
of $\widetilde{T^{\vee}}$ for which $\rho^{\fs}(X_j) = Z_j, \; j = 1,
\ldots, 5$.
Let $[z_1,z_2]$ denote the image of $(z_1,z_2)$ via the standard
quotient map $p^\fs\colon T^{\vee} \to T^{\vee}/W^\fs$, so that 
$[z_1,z_2] = W^\fs \cdot (z_1,z_2)$.
When the first column itemizes the pairs
$(w,t) \in \widetilde{T^{\vee}}$, the second column
itemizes $p^\fs(h_j(\tau)t)$. The third column itemizes the corresponding
correcting cocharacters.
\[
\begin{array}{lll}
X_1 = ((ab)^2,(1,1))  &  [1,\tau^{-2}1] & h_1(\tau) =  (1,\tau^{-2}) \\
X_2 = ((ab)^2,(j,j^2))  & [j,\tau^{-2}j^2] & h_2(\tau) =  (1,\tau^{-2})   \\
X_3 = ((ab)^2,(j^2,j))  & [j^2,\tau^{-2}j] & h_3(\tau) =  (1,\tau^{-2}) \\
X_4 =  \{(a,(z,z)): z \in \CC^{\times}\}  & \{[z,\tau^{-2}z]: z \in
\CC^{\times}\}  & h_4(\tau) =  (1,\tau^{-2})\\  
X_5 = \{(e,(z_1,z_2)):z_1,z_2\in\Cset^\times\}
&\{[z_1,z_2]:z_1,z_2\in\Cset^\times\} & h_5(\tau) = 1
\end{array}
\]

It is now clear that  cocharacters can assigned to two-sided cells as
follows:
$h_{\bc}(\tau)  =  (1,\tau^{-2})$ if $\bd = \bd_1, \bd_e$, and
$h_{\bd_0} = 1$. 

The map $\pi_{\sqrt q}$ sends the two distinct points $(1/\sqrt q,1/\sqrt q)$
and $(\sqrt q,\sqrt q)$ in $D^a/\Cent_{W^\fs}(a)$, the affine line
attached to the cell $\bd_1$, to the one point $y_a\in
D/W^\fs$ since $(1,q^{-1})$, $(q,1)$ are in the same $W^\fs$-orbit:
$(1,q^{-1})\cong_{baba}(q,1)$. 

We have
\[\pi^{-1}_{\sqrt q}(y_a)=\left\{\rho^\fs(e,(1,q^{-1})),
\rho^\fs((ab)^2,(1,1)),\rho^\fs(a,(1,1)),\rho^\fs(a,(q,q))\right\}.\]
It follows that
\[|\pi^{-1}_{\sqrt q}(y_a)|=4=|i_{T\subset B}^G(y_a)|.\]

Similarly we obtain that
$|\pi^{-1}_{\sqrt q}(y_a')|=|\pi^{-1}_{\sqrt
q}(y_a''|=4=|i_{T\subset B}^G(y_a'')|=|i_{T\subset B}^G(y_a')|$.
\end{proof}

\begin{lem} \label{lem:bijSL}  Part (4) of Theorem \ref{thm:main}  is true for the point
\[\fs=[T,\chi\otimes\chi]_G\in \mathfrak{B}(\rG_2)\] when $\chi$
is a cubic ramified character of $F^\times$.
\end{lem}

\begin{proof}
The semisimple elements $y$, $\sigma$ below are always related as
follows:\[ \sigma = \pi_{\sqrt q}(y).\]
Let $\eta^\fs\colon(T^\vee//W^\fs)\to\Irr(J^\fs)$ be the bijection which
is induced by the Morita equivalences in~(\ref{eqn:MoritaSL}).
Then the definition~(\ref{eqn:mu_c}) of $\mu^\fs\colon
(T^\vee\q W^\fs)\to \Irr(G)^\fs$ gives:
\[\mu^\fs(y) =\begin{cases}
\cV_{\sigma,1,1}^\fs\,,&\text{if  $y\in T^\vee/W^\fs$;}\cr
\cV_{\sigma,u_1,1}^\fs\,,&\text{if $y\in
(T^\vee/W^\fs)_{\bd_1}$;}\end{cases}
\]
the three isolated points are sent to the $L$-indistinguishable elements in the
discrete series which has inertial support $\fs$:
\[\mu^\fs(pt_1) = \cV_{y_a,u_e,\rho_1}^\fs,\quad \mu^\fs(pt_2) =
\cV_{y_a',u_e,\rho_2}^\fs, \quad
\mu^\fs(pt_3) =\cV_{y_a'',u_e,\rho_1}^\fs.\]
Now the infinitesimal character of $\mathcal{V}_{\sigma,u,\rho}^\fs$
is $\sigma$, therefore the map $\mu^\fs$ satisfies
\[inf.ch. \circ \mu^\fs = \pi_{\sqrt q}.\]
\end{proof}

\begin{lem} \label{lem:tempSL} Part (5) of Theorem \ref{thm:main} is true
in this case.\end{lem}
\begin{proof}
It follows from \cite[Prop. 1.1, p. 469]{M} that $\delta(\chi)$ (viewed as a
representation of $M_\alpha$) is the unique subrepresentation of
$I^\alpha(\nu^{1/2}\chi\otimes\nu^{-1/2}\chi)$. So it has inertial
support $[T,\nu^{1/2}\chi\otimes\nu^{-1/2}\chi]_{M_\alpha}$. It
implies that $I_\alpha(0,\delta(\chi))$ has inertial support
$[T,\nu^{1/2}\chi\otimes\nu^{-1/2}\chi]_G=[T,\chi\otimes\chi]_G=\fs$.

If we look at $M_\beta$, still from \cite[Prop. 1.1, p. 469]{M},
we see that $\delta(\chi)$ (here viewed as a representation of
$M_\beta$) is the unique subrepresentation of
$I^\beta(\nu^{-1/2}\chi\otimes\nu)$. Hence it has inertial support
$[T,\nu^{-1/2}\chi\otimes\nu]_{M_\beta}$. It follows that
$I_\beta(0,\delta(\chi))$ has inertial support
$[T,\nu^{-1/2}\chi\otimes\nu]_G=[T,\chi\otimes 1]_G$, which is not
equal to $\fs$, because $\chi\otimes 1$ does not belong to the
$W$-orbit of $\chi\otimes\chi$ (see also Proposition~\ref{prop:inertialsupports}).

Hence the compact extended quotient is accounted for as
follows: The compact quotient $E/W$ is sent to
the unitary principal series \[\{I(\psi_1\chi \otimes
\psi_2\chi: \psi_1, \psi_2 \in \Psi(F^{\times})\}/W^{\fs},\]
the component $U(1)$ to
the intermediate unitary series \[\{I_{\alpha}(0,
\delta(\psi\chi)): \psi \in \Psi^{\temp}(F^{\times})\},\] and the $3$ isolated points
$pt_1$, $pt_2$, $pt_3$ are sent to the $3$
elements in the discrete series\[\pi(\chi) \subset I(\nu\chi
\otimes \chi),\quad \pi(\omega\chi) \subset I(\nu\omega\chi
\otimes \omega\chi),\quad \pi(\omega^2\chi) \subset
I(\nu\omega^2\chi \otimes \omega^2\chi).\]
\end{proof}

\section{The case $H^\fs=\SO(4,\Cset)$} \label{sec:SO}
We assume in this section that
$\chi_1=\chi_2=\chi$ with $\chi$ a ramified quadratic character.
It follows from~(\ref{inertie}) that
\[\fs=[T,\chi\otimes\chi]_G=[T,\chi\otimes 1]_G=[T,1\otimes \chi]_G.\]
From~(\ref{elements}), we get
\begin{eqnarray}
\label{eqn:four}
W^{\fs}&=&\left\{e,a,babab,bababa\right\}
\,=\,\left\{e,a,r^3,ar^3\right\}\\
&=&\langle s_\alpha, s_{3\alpha + 2\beta}\rangle
\,\cong\, \Zset/2\Zset \times \Zset/2\Zset.\nonumber\end{eqnarray}
We recall that $Q^\fs\subset X(T)$ denotes the root lattice of $\Phi^\fs$.
We have
\begin{equation}
Q^\fs=\Zset\alpha \oplus \Zset(3\alpha+2\beta).\end{equation}
Hence $Q^\fs$ is strictly contained in $X(T)$, see~(\ref{XT}).
This shows that the group $H^\fs$ is not simply connected. Setting
$V:=Q^\fs\otimes_{\Zset}\Qset$, we define the weight lattice $P^\fs$, as
in \cite[chap.~VI, 1.9]{Bou}, by $$P^\fs:=\left\{x\in V \,:
\,<x,\gamma>\in\Zset,\,\forall \gamma\in\Phi^{\fs\vee}\right\}.$$ We have
$$P^\fs=\frac{1}{2}\Zset\alpha \oplus
\frac{1}{2}\Zset(3\alpha+2\beta).$$ Hence $X(T)$ is strictly contained
in $P^\fs$. This shows that the $H^\fs$ is not of adjoint type.
Now, let $X_0$ denote the subgroup of $X(T)$
orthogonal to $\Phi^{\fs\vee}$. We see that $X_0=\{0\}$. This
means that the group $H^\fs$ is semisimple.
Hence $H^\fs$ is isomorphic to $ \SO(4,\Cset)$. The group $\SO(4,\Cset)$ is
isomorphic to the quotient group
$(\SL(2,\Cset) \times \SL(2,\Cset))/\langle -\Id,-\Id\rangle$, where $\Id$
is the identity in $\SL(2,\Cset)$.

The group $H^\fs$ admits $4$ unipotent classes $\be_0$, $\be_1$, $\be_1'$,
$\be_e$. The closure order on unipotent classes is the
following:
\[\vcenter{\xymatrix@R=5pt@C=2.5pt{
&\be_e \ar@{-}[dl] \ar@{-}[dr] \\
\be_1 \ar@{-}[dr] && \be_1' \ar@{-}[dl] \\
&\be_0}}.\]
We have \[J^\fs=J_{\be_e}^\fs\oplus J_{\be_1}^\fs\oplus J_{\be'_1}^\fs\oplus J_{\be_0}^\fs.\]

\smallskip

\par
\noindent
{\bf Definition.}
We define the following partition of $T^\vee\q W^\fs$:
\begin{equation} \label{defn:extSO}
T^\vee\q W^\fs=(T^\vee\q W^\fs)_{\be_e}\sqcup(T^\vee\q
W^\fs)_{\be_1} \sqcup(T^\vee\q W^\fs)_{\be'_1}\sqcup(T^\vee\q
W^\fs)_{\be_0},\end{equation} where
\begin{eqnarray*}
(T^\vee\q W^\fs)_{\be_e}&:=&pt_1\sqcup pt_2,\\
(T^\vee\q W^\fs)_{\be_1}&:=&(T^\vee)^{a}/\Cent(a)\cong\Aset^1,\\
(T^\vee\q
W^\fs)_{\be_1'}&:=&(T^\vee)^{ar^3}/\Cent(ar^3)\cong\Aset^1,\\
(T^\vee\q W^\fs)_{\be_0}&:=&T^\vee/W\sqcup pt_*,
\end{eqnarray*}
with $pt_1:=(1,1)$, $pt_2:=(-1,-1)$ and $pt_*:=(1,-1) \sim_{W^\fs} (-1,1)$.

\smallskip
\smallskip

{\sc Note.} On the contrary to the definitions in the previous cases, 
the above definition does not correspond to a partition of the conjugacy
classes in $W$ indexed by unipotent classes in $G^\vee$. Indeed, we get the
following correspondence between conjugacy classes in $W$ and
unipotent classes in $G^\vee$:
\begin{itemize}
\item $(a)$ corresponds to the unipotent class in
$G^{\vee}$ corresponding to $\be_1$,
\item $(ar^3)$ corresponds to the unipotent
class in $G^{\vee}$ corresponding to $\be_1'$,
\item $(e)$ corresponds to the
class in $G^{\vee}$ corresponding to $\be_0$, 
\item $(r^3)$ corresponds to {\bf both} the unipotent class in $G^{\vee}$ corresponding to
$\be_e$ and the unipotent class in $G^{\vee}$ corresponding to $\be_0$.
\end{itemize}

\begin{lem} \label{lem:extSO} We have
\begin{eqnarray*}
T^\vee\q W^\fs &=& (pt_1 \sqcup pt_2)
\sqcup \mathbb{A}^1 \sqcup \mathbb{A}^1 \sqcup (T^\vee/W^\fs\sqcup pt_*)\\
E^\fs\q W^\fs &=& (pt_1 \sqcup pt_2) \sqcup
\mathbb{I} \sqcup \mathbb{I} \sqcup (E^\fs/W^\fs\sqcup pt_*).\end{eqnarray*}
Moreover,
we have a ring isomorphism \[\Cset[T^\vee/W^\fs] \sim \Cset[X,Y]_0,\]
where $\Cset[X,Y]_0$ denotes the coordinate ring of
the quotient of $\Aset^2$ by the action of $\Zset/2\Zset$ which reverses
each vector.

We have $J^\fs\asymp \mathcal{O}(T^{\vee}\q W^\fs)$, where
\begin{eqnarray*}
J_{\be_e}^\fs\sim_{morita}\cO((T^\vee\q W^\fs)_{\be_e}),&
J_{\be_1}^\fs\sim_{morita}\cO((T^\vee\q W^\fs)_{\be_1}),\\
J_{\be_1'}^\fs\sim_{morita}\cO((T^\vee\q W^\fs)_{\be_1'}),&
J_{\be_0}^\fs\asymp\cO((T^\vee\q W^\fs)_{\be_0}).
\end{eqnarray*}
\end{lem}
\begin{proof}
1. Extended quotient:
Let $D = D^\fs\cong T^\vee$. We give the case-by-case analysis.
\begin{itemize}
\item $c=e$.  $D^c/\Cent(c) = D/W^{\fs}$.
\item $c=a$.  $D^c= \{(t,t) : t \in \Cset^\times\}$.
\[D^c/\Cent(c) =
\{\{(t,t),(t^{-1},t^{-1})\}: t \in \Cset^\times\} \cong\Aset^1.\]
\item $c=r^3$.   $D^c = \{(1,1), (1,-1), (-1,1), (-1,-1)\}$.
Therefore
$D^c/\Cent(c) = pt_1 \sqcup pt_2 \sqcup pt_*$.
\item $c=ar^3$.   $D^c
= \{(t,t^{-1}): t \in \Cset^\times\}$.
\[D^c/\Cent(c) =
\{\{(t,t^{-1}),(t^{-1},t)\}: t \in \Cset^\times\}\cong\Aset^1.\]
\end{itemize}

Let $\Mset[u]:=\Cset[u,u^{-1}]$ denotes the $\mathbb{Z}/2\mathbb{Z}$-graded
algebra of Laurent
polynomials in one indeterminate $u$.   Let $a$ denote the
generator of $\Zset/2\Zset$.  The group $\Zset/2\Zset$ acts as
automorphism of $\mathbb{M}[u]$, with $a(u) = u^{-1}$.  We
define
\[\mathbb{L}[u]: = \{P\in \mathbb{M}[u]\,:\, a(P) = P\}\] as the algebra
of balanced Laurent polynomials in $u$.

Let $\uT^\vee$ be the maximal torus of $\SL(2,\Cset)\times\SL(2,\Cset)$.
Then the coordinate ring $\Cset[\uT^\vee/W^\fs]$ is $\Lset[u]\otimes \Lset[v]$. The map
\[(u+1/u,v+1/v) \mapsto  (X,Y)\]
sends $\Cset[\uT^\vee/W^\fs]$ to $\Cset[X,Y]$ (the polynomial algebra in
two indeterminates $X$, $Y$), a ring isomorphism.
The coordinate ring of an affine plane $\Aset^2$.

Recall that $T^\vee$ is the standard maximal torus in
$H^\fs\cong(\SL(2,\Cset) \times\SL(2,\Cset)) / \langle-\Id,-\Id\rangle$. Hence
it follows that $\Cset[T^\vee/W^\fs]$ is the ring of
balanced polynomials in $u$, $v$ which are fixed under
$(u,v) \mapsto (-u,-v)$.
These polynomials correspond to those polynomials in $X$, $Y$
which are fixed under $(X,Y)\mapsto (-X,-Y)$. Therefore we have a
ring isomorphism $\Cset[T^\vee/W^\fs] \sim \Cset[X,Y]_0$.

2. Compact extended quotient:
\[E^{\mathfrak{s}}/W^{\mathfrak{s}} \sqcup \mathbb{I} \sqcup
\mathbb{I} \sqcup pt \sqcup pt \sqcup pt = E^{\mathfrak{s}}\q
W^{\mathfrak{s}}.\]
The group $\tW_\aff^\fs$ (see~(\ref{eqn:Wsa})) is the extended affine
Weyl group of
the $p$-adic group $(\SL(2,F)\times\SL(2,F))/\langle-\Id,-\Id\rangle$, which
admits $H^\fs$ as Langlands dual.

3. Extended affine Weyl group:

We will describe the group $\tW_\aff^\fs$. 
Let $\uX$ be the cocharacter group of a maximal torus of $\PGL(2,F)\times\PGL(2,F)$, and let
$\uW^\fs_\aff:=W^\fs \ltimes \uX$. Then $\tW_\aff^\fs$ is a subgroup of
$\uW^\fs_\aff$. 

Let $W_2$ be the
extended affine Weyl group corresponding to $\PGL(2,F)$, that is
$W_2=\Zset/2\Zset\ltimes X_2$, where $X_2$ is the cocharacter of a
maximal torus of $\PGL(2,F)$ and $\Zset/2\Zset=\{e,a\}$. Let $t\ne
a$ be the other simple reflection in $W_2$. In $W_2$ there exists
a unique element $g$ of order $2$ such that $gag=t$. It is known
that the length of $g$ is $0$. Let $W'_2$ be a copy of $W_2$. The
simple reflections in $W_2'$ will be denoted by $a'=ar^3$, $t'$
correspondingly. Denote by $g'$ the element corresponding to $g$,
then $(g')^2=e$ and  $g'a'g'=t'$. Then $\uW^\fs_\aff=W_2\times W'_2$ and 
$\tW_\aff^\fs$ is the
subgroup of $\uW^\fs_\aff$ generated by $gg'$, $a$, $t$, $a'$,
$t'$.

4. Asymptotic Hecke algebra:
Since $\tW_\aff^\fs$ is a subgroup of $\uW^\fs_\aff$, its based ring of 
$\tW_\aff^\fs$ can be described as a subring of the based ring
of the latter. 

$\bullet$ From the above description, we see that the two-sided cell $\be_e$
of $\tW_\aff^\fs$ consists of $e$ and $gg'$ and that
its based ring is isomorphic
to the group algebra of $\Zset/2\Zset$. It follows that
$J_{\be_e}$ is Morita equivalent to $\Cset \oplus \Cset$. Hence we have
\begin{equation}
\label{eqn:JSOe}
J_{\be_e}^\fs\sim_{morita}\cO((T^\vee\q W^\fs)_{\be_e}).
\end{equation}

$\bullet$ Let $U_2$ be the subgroup of $\tW_\aff^\fs$ generated by $gg'$, $a$,
$t$, then the map which sends $gg'$ to $g$, $a$ to $a$, $t$ to $t$ defines
an isomorphism from $U_2$ to $W_2$. Thus $\be_1$ is
the lowest two-sided cell of $U_2$, which equals $U_2-\{e,gg'\}$. Therefore the
based ring of $\be_1$ is isomorphic to $\Mat_2(R)$, where $R$ denotes the
representation ring of $\SL(2,\Cset)$. Hence the
based ring of $\be_1$ is clearly Morita equivalent to $R$, but $R$ is
isomorphic to the polynomial ring $\Zset[u]$ in one variable.
It follows that $J_{\be_1}$ is Morita equivalent to $\Cset[u]$, where $u$
is an indeterminate. From the first part of the Lemma, we get
that
\begin{equation}
\label{eqn:JSOun}
J_{\be_1}^\fs\sim_{morita}\cO((T^\vee\q W^\fs)_{\be_1}).
\end{equation}

$\bullet$ Similarly, let $U'_2$ be the subgroup of $\tW_\aff^\fs$ generated by
$gg'$, $a'$, $t'$, then the map which send $gg'$ to $g$, $a'$
to $a$, $t'$ to $t$ defines an isomorphism from $U'_2$ to $W_2$.
Thus $\be'_1$ is the
lowest two-sided cell of $U'_2$, which equals $U_2'-\{e,gg'\}$. So the based
ring of $\be'_1$ is isomorphic to $\Mat_2(R)$,
which is also Morita equivalent to the polynomial ring $\Zset[v]$ in
one variable. From the first part of the Lemma, we obtain
that
\begin{equation}
\label{eqn:JSOunprime}
J_{\be_1'}^\fs\sim_{morita}\cO((T^\vee\q W^\fs)_{\be_1'}).
\end{equation}

$\bullet$
Let $x$ denote the
fundamental cocharacter of $X_2$ and write the operation in $X_2$ by
multiplication. For the fundamental cocharacter $x$, we still use $x$ if
it is
regarded as an element in $W_2$ and denote it by $x'$  if it is regarded
as an
element in $W'_2$. Let $a$ be the non-unit element of $\Zset/2\Zset$.  For
the element $a$,
we still use $a$ if it is regarded as an element in $W_2$ and denote it
by $a'$ if
it is regarded as an element in $W'_2$. In this way $\tW_\aff^\fs$ is the
subgroup
of $\uW^\fs_\aff$ which
consists of the elements $(a^m x^i)(a^{\prime n} x^{\prime j})$
with $m+n$ even and $i,j\in\{0,1\}$.

\smallskip

The lowest two-sided cell $\be_0$ of $\tW_\aff^\fs$
consists of the $16$ elements which are listed in the first column of
Table~\ref{tab:J}. 
(Removing the restriction on $m+n$, these elements form the lowest
two-sided cell $\bar\be_0$ of $\uW_\aff^\fs$).

Let $V(\ell)$ be the irreducible
representation of $\SL_2(\Cset)$ of highest weight $\ell$ and let $V_{ij}(\ell)$
be the element
in $\Mat_2(R)$ whose $(i,j)$ entry is $V(\ell)$ and other entries are $0$. Then
it follows from \cite[Theorem~1.10 and its proof in~\S~4.4]{Xic0} that the based 
ring of $\bar\be_0$ is isomorphic to $\Mat_2(R)\otimes \Mat_2(R)$,
and that the element in $\Mat_2(R)\otimes\Mat_2(R)$, denoted by $V_w$, 
corresponding to $t_w$, for $w\in \be_0$, is given by Table~\ref{tab:J}.
\begin{table}
\begin{center}
\begin{tabular}{|l|l|} 
\hline
$\quad$ $\quad$ $\quad$ $\quad$ $\quad$ $w\in\be_0$&$\quad$ $\quad$ $\quad$
$\quad$ $V_w$  \\
\hline
$(ax^m)(a'x^{\prime n})$, $m,n\ge 0$, $m+n$ even & $V_{11}(m)\otimes V_{11}(n)$\\ \hline
$(ax^m)(x^{\prime n})$, $m\ge 0$, $n\ge 1$,  $m+n$ even&$V_{11}(m)\otimes
V_{21}(n-1)$ \\ \hline
$(ax^m)(a'x^{\prime n})$, $m\ge 0$,  $n\le -2$, $m+n$ even&$V_{11}(m)\otimes
V_{22}(-n-2)$ \\ \hline
$(ax^m)(x^{\prime n})$, $m\ge 0$,  $n\le -1$, $m+n$ even &
$V_{11}(m)\otimes V_{12}(-n-1)$ \\ \hline
$(x^m)(a'x^{\prime n})$, $m\ge 1$, $n\ge 0$, $m+n$ even
&$V_{21}(m-1)\otimes V_{11}(n)$\\\hline
$(x^m)(x^{\prime n})$, $m\ge 1$, $n\ge 1$,  $m+n$ even
&$V_{21}(m-1)\otimes V_{21}(n-1)$ \\\hline
$(x^m)(a'x^{\prime n})$, $m\ge 1$,  $n\le -2$, $m+n$ even
&$V_{21}(m-1)\otimes V_{22}(-n-2)$ \\\hline
$(x^m)(x^{\prime n})$, $m\ge 1$,  $n\le -1$, $m+n$ even
&$V_{21}(m-1)\otimes V_{12}(-n-1)$ \\\hline
$(ax^m)(a'x^{\prime n})$, $ m\le -2$, $n\ge 0$, $m+n$ even&
$V_{22}(-m-2)\otimes V_{11}(n)$ \\ \hline
 $(ax^m)(x^{\prime n})$, $ m\le -2$, $n\ge 1$,  $m+n$ even&
$V_{22}(-m-2)\otimes V_{21}(n-1)$ \\ \hline
$(ax^m)(a'x^{\prime n})$,  $m\le -2$,  $n\le -2$, $m+n$ even&
$V_{22}(-m-2)\otimes V_{22}(-n-2)$ \\ \hline
$(ax^m)(x^{\prime n})$,   $m\le -2$,  $n\le -1$, $m+n$ even&
$V_{22}(-m-2)\otimes V_{12}(-n-1)$ \\ \hline
$(x^m)(a'x^{\prime n})$, $m\le -1$, $n\ge 0$, $m+n$ even&
$V_{12}(-m-1)\otimes V_{11}(n)$ \\ \hline
$(x^m)(x^n)$, $m\le -1$, $n\ge 1$,  $m+n$ even&
$V_{12}(-m-1)\otimes V_{21}(n-1)$ \\ \hline
$(x^m)(a'x^{\prime n})$, $m\le -1$,  $n\le -2$, $m+n$ even&
$V_{12}(-m-1)\otimes V_{22}(-n-2)$ \\ \hline
$(x^m)(x^{\prime n})$,  $m\le -1$,  $n\le -1$, $m+n$ even&
$V_{12}(-m-1)\otimes V_{12}(-n-1)$ \\ \hline
\end{tabular}
\caption{\label{tab:J} the map $w\mapsto V_w$.}
\end{center}
\end{table}

Then the based ring of $\be_0$ is isomorphic to the subring of
$\Mat_2(R)\otimes_\Zset\Mat_2(R)$
spanned by the elements
$(\tau_{ij}V(m_{ij}))\otimes (\tau'_{kl}V(n_{kl}))$,  with condition
$m_{ij}+n_{kl}+i+j+k+l$ is even, where all $\tau_{ij}$ and $\tau'_{kl}$
are
integers. Hence $J_{\be_0}$ is isomorphic to the subring of
$\Mat_2(\Cset[X])\otimes_\Cset \Mat_2(\Cset[Y])$ spanned by the elements
$(\tau_{ij}z^{m_{ij}})\otimes (\tau'_{kl}z^{n_{kl}})$,  with condition
$m_{ij}+n_{kl}+i+j+k+l$ is even, where all $\tau_{ij}$ and $\tau'_{kl}$
are integers.

The above parity condition is: $m_{ij} + b_{kl} + i + j + k + l$ even.  For
example, $m_{11} + n_{11}$ is even, $m_{22} + n_{11}$ is even, $m_{12} +
n_{11}$ is odd, $m_{21} + n_{11}$ is odd, $m_{12} + n_{12}$ is even, $m_{21} +
n_{21}$ is even, $m_{22} + n_{22}$ is even; so we are allowed monomials of even
degree on the diagonals of both matrices, monomials of odd degree on the
off-diagonals of both matrices OR the same thing with reversed parity. Taking
the span, we realize all even polynomials on the diagonal, all odd polynomials
on the off-diagonal in both matrices OR the same thing with reversed parity.

\smallskip

In other words, the ring $\Mat_2(\Cset[X])$ is $\Zset_2$-graded:
\begin{itemize}
\item[--]
$(\Mat_2(\Cset[X]))_0$: $=$ even polynomials on the diagonal,
odd polynomials on the off-diagonal;
\item[--]
$(\Mat_2(\Cset[X]))_1$: $=$ odd polynomials on the diagonal, even polynomials
on the off-diagonal.
\end{itemize}
Consider the $\Zset_2$-graded tensor product
$\Bset[X,Y]:=\Mat_2(\Cset[X])\otimes_\Cset \Mat_2(\Cset[Y])$.
Then $J_{\be_0}$ is isomorphic to the even part $\Bset[X,Y]_0$ of
$\Bset[X,Y]$.

Give $\Cset[X,Y]$ a $\Zset_2$-grading by the convention that a monomial
$X^mY^n$ is even (odd) according to the parity of $m+n$. Form the
algebra $\Mat_4(\Cset[X,Y])$. Give this a $\Zset_2$-grading by saying that
the even (resp. odd) elements are those which have a $2\times 2$-block in
the upper left corner consisting of even (resp. odd) polynomials, a
$2\times 2$-block in the lower right corner consisting of even (resp. odd)
polynomials, a $2\times 2$-block in the lower left corner consisting of odd
(resp. even) polynomials, a $2\times 2$-block in
the upper right corner consisting of odd (resp. even) polynomials.

Then the even part of $\Mat_2(\Cset[X]) \otimes\Mat_2 (\Cset[Y])$ is isomorphic to
the even part of $\Mat_4(\Cset[X,Y])$ ---  \ie as $\Zset_2$-graded algebras
$\Mat_2(\Cset[X]) \otimes \Mat_2(\Cset[Y])$ and $\Mat_4 (\Cset[X, Y])$ are
isomorphic.


Let $\Mat_4( \Cset[X, Y])_0$ consist of all $4\times 4$-matrices
with entries in $\Cset[X, Y]$ such that: the upper left $2\times
2$ block and the lower right $2\times 2$ block are $2\times 2$
matrices with entries in $\Cset[X, Y]_0$ and the lower left
$2\times 2$ block and the upper right $2\times 2$ block are
$2\times 2$ matrices with entries in $\Cset[X, Y]_1$. Let
$\uP(X,Y)=(P_{i,j}(X,Y))_{1\le i,j\leq 4}$ be an element of
$\Mat_4(\Cset[X,Y])$. We write $\uP(X,Y)$ as a $2\times 2$ block
matrice as
\[\uP(X,Y)=\left(\begin{matrix} \uP(X,Y)_{1,1}&\uP(X,Y)_{1,2}\cr
\uP(X,Y)_{2,1}&\uP(X,Y)_{2,2}\end{matrix}\right),\]
where $\uP(X,Y)_{i,j}\in\Mat_2(\Cset[X, Y])$, for $i,j\in\{1,2\}$. Hence
the matrix $\uP(X,Y)$ is in $\Mat_4(\Cset[X,Y])_0$ if and only if
we have, for $i,j\in\{1,2\}$,
\[\text{$\uP(X,Y)_{i,i}\in\Mat_2(\Cset[X, Y]_0)$ and
$\uP(X,Y)_{i,j}\in\Mat_2(\Cset[X, Y]_1)$ if $i\ne j$.}\]

If $(z,z')$ is a pair of complex numbers, then evaluation at $(z,z')$
gives an algebra homomorphism
\begin{eqnarray*}
\ev_{(z,z')}\colon &\Mat_4 ( \Cset[X, Y])_0\longrightarrow \Mat_4(\Cset)\cr
&\uP(X,Y)\mapsto \uP(z,z').
\end{eqnarray*}
The algebra homomorphism $\ev_{(z,z')}$ is surjective except when $(z, b)=(0,
0)$. In the case $(z,z')=(0, 0)$ the image of $\ev_{(z,z')}$
is the subalgebra of $\Mat_4(\Cset)$ consisting of
all $4\times 4$ matrices of complex numbers such that the lower
left $2\times 2$ block and the upper right $2\times 2$ block are
zero --- \ie the image is $\Mat_2(\Cset)\oplus\Mat_2(\Cset)$
embedded in the usual way in $\Mat_4(\Cset)$. So except for $(z,
z')=(0,0)$ we have a simple module, say $M_{(z,z')}$. When $(z,z')=(0, 0)$
we have a
module which is the direct sum of two simple modules, that we denote by
$M'_{(0,0)}$ and $M''_{(0,0)}$.

Since we have
\begin{eqnarray*}
\left(\begin{matrix}
\II_2&0\cr
0&-\II_2\end{matrix}\right)\uP(z,z')\left(\begin{matrix}
\II_2&0\cr
0&-\II_2\end{matrix}\right)^{-1}=&
\left(\begin{matrix} \uP(z,z')_{1,1}&-\uP(z,z')_{1,2}\cr
-\uP(z,z'b)_{2,1}&\uP(z,z')_{2,2}\end{matrix}\right)\cr
=&\left(\begin{matrix} \uP(-z,-z')_{1,1}&\uP(-z,-z')_{1,2}\cr
\uP(-z,-z')_{2,1}&\uP(-z,-z')_{2,2}\end{matrix}\right),\end{eqnarray*}
the matrix $\left(\begin{matrix}
\II_2&0\cr
0&-\II_2\end{matrix}\right)$ conjugates the simple modules $M_{(z,z')}$ and
$M_{(-z,-z')}$.

On the other hand, let $(z_1,z'_1)$ be a pair of complex numbers such that
$(z_1,z'_1)\notin\{(z,z'),(-z,-z')\}$. Then there exists an even po\-ly\-no\-mial $Q(X,Y)$
such that $Q(z_1,z'_1)\neq Q(z,z')$. Indeed, 
\begin{itemize} \item
if $z_1\notin\{z,-z\}$, we can take
$Q(X,Y)=X^2$,
\item if $z_1=z$, we can take $Q(X,Y)=XY$ (since then $z_1'\ne z'$),
\item if $z_1=-z$, we can take $Q(X,Y)=XY$ (since then $z_1'\ne -z'$).
\end{itemize}

Consider the matrix
\[\uQ(X,Y):=\left(\begin{matrix} Q(X,Y)&0&0&0\cr
0&0&0&0\cr
0&0&0&0\cr
0&0&0&0\cr\end{matrix}\right).\]
We have $\ev_{(z,z')}(\uQ(X,Y))\neq\ev_{(z_1,z'_1)}(\uQ(X,Y))$. It follows that
the simple modules $M_{(z,z')}$ and $M_{(z_1,z'_1)}$ are not isomorphic.

\smallskip

Let $A=\Mat_4(\Cset[X,Y])$ and let $\Gamma:=\{1,\varepsilon\}$, where
$\varepsilon\colon A\to A$ is defined by
\[\varepsilon(\uP(X,Y))=\left(\begin{matrix}
\hphantom{-}\uP(-X,-Y)_{1,1}&-\uP(-X,-Y)_{1,2}\cr
-\uP(-X,-Y)_{2,1}&\hphantom{-}\uP(-X,-Y)_{2,2}\end{matrix}\right).\]
From Lemma~\ref{lem:crossedproduct}, we know
that the unital $\Cset$-algebras $A^\Gamma$ and
$(A\rtimes\Gamma)e_\Gamma(A\rtimes\Gamma)$ are Morita equivalent.
Here $e_\Gamma=\frac{1}{2}(1+\varepsilon)$.

We embed $A$ into the crossed product algebra $A\rtimes\Gamma$ by
sending $\uP(X,Y)$ to $\uP(X,Y)[1]$.
For $1\le i,j\le 4$, let $E_{i,j}\in A$ be the matrix with entry $(i,j)$
equal to $1$ and all the other entries equal to $0$. We have
\[E_{3,1}([1]+[\varepsilon])=E_{3,1}[1]+E_{3,1}[\varepsilon]\;\;
\text{and}\;\;
([1]+[\varepsilon])E_{3,1}=E_{3,1}[1]-E_{3,1}[\varepsilon].\]
We get \[E_{3,1}[1]=(\frac{1}{2}E_{3,1})([1]+[\varepsilon])+
([1]+[\varepsilon])(\frac{1}{2}E_{3,1}),\]
it follows that
$E_{3,1}[1]$ belongs to the two-sided ideal $([1]+[\varepsilon])$.

Since $E_{i,1}[1]=E_{i,3}(E_{3,1}[1])$, it gives that
$E_{i,1}[1]\in([1]+[\varepsilon])$ for each $i$. Since
$E_{i,j}[1]=(E_{i,1}[1])E_{1,j}$, then we get that
$E_{i,j}[1]\in([1]+[\varepsilon])$ for any $i,j$.
Hence we have proved that
\[(A\rtimes\Gamma)e_\Gamma(A\rtimes\Gamma)=A\rtimes\Gamma.\]
It follows that the unital $\Cset$-algebras $A^\Gamma$ and
$A\rtimes\Gamma$ are Morita equivalent.

Thus we have proved that for $\Mat_4
( \Cset[X, Y])_0=A^\Gamma$ the $M_{z,z'}$ with $(z,z')\ne(0,0)$, and $M'_{(0,0)}$,
$M''_{(0,0)}$ are (up to isomorphism) all the simple
modules and that they are distinct except that $M_{(z,z')}$ and $M_{(-z,-z')}$
are isomorphic.


Let $\cJ$ be the ideal in $\Mat_4 ( \Cset[X, Y])_0$ which (by
definition) is the pre-image with respect to $\ev_{(0,0)}$
of $\Mat_2(\Cset)\oplus \{0\}$. Then $\ev_{(z,z')}$
surjects $\cJ$ onto $\Mat_4(\Cset)$ except at $(z,z')=(0, 0)$, and
$\ev_{(0, 0)}$ surjects $\cJ$ onto $\Mat_2(\Cset)\oplus
\{0\}$.
In $\Cset[X, Y]_0\oplus\Cset$  let $\cI$ be the ideal $\Cset[X,
Y]_0\oplus \{0\}$. Consider the two filtrations
\[\{0\}\subset \cJ
\subset \Mat_4 (\Cset[X, Y])_0\]
\[\{0\} \subset \cI\subset\Cset[X, Y]_0\oplus \Cset.\]
Let $\delta_0$ be the algebra homomorphism
\begin{eqnarray*}
\delta_0\colon&\Cset[X, Y]_0\oplus\Cset\longrightarrow\Mat_4 (\Cset[X, Y])_0\cr
&(P(X,Y),z)\mapsto\left(\begin{matrix}P(X,Y)&0&0&0\cr
0&0&0&0\cr
0&0&z&0\cr
0&0&0&0\end{matrix}\right),
\end{eqnarray*}
where the complex number $z$ is viewed as an even polynomial of degree
zero.

We view the ideal $\cJ$ as an algebra.
We recall that $k=\Cset[X, Y]_0$. Then $k$ is a unital finitely generated
free commutative algebra. Hence $k$ is the coordinate algebra of
an affine variety, the variety $\Cset^2/(z,z')\sim(-z,-z')$, thus $k$ is
noetherian. It follows that $B:=\Mat_4(\Cset[X,Y]_0)$ is a $k$-algebra of
finite type. This implies that $\cJ$ is a $k$-algebra of
finite type. Therefore any simple $\cJ$-module, as a vector space over
$\Cset$ is finitely dimensional.

Hence any simple $\cJ$-module gives a surjection
\[\cJ\twoheadrightarrow \Mat_n(\Cset),\]
and, by using Lemma~\ref{lem:rings}, extends uniquely to a simple
$B$-module.

It follows that $\delta_0$ is spectrum preserving with respect to these
filtrations.

This proves that $J^\fs_{\be_0}$ is geometrically equivalent to
$\Cset[X, Y]_0 \oplus \Cset$. This is the coordinate ring of the
extended quotient $(\Aset^2)\q (\Zset/2\Zset)$, see
first part of the Lemma. Hence we have
\begin{equation} \label{eqn:SOJzero}
J_{\be_0}^\fs\asymp\cO((T^\vee\q W^\fs)_{\be_0}).
\end{equation}
\end{proof}

{\sc Note.} When the complex reductive group is simply-connected,
we show in \cite[Theorem 4]{ABP} that we have
\[J_{\bc_0} \asymp \mathcal{O}(T^{\vee}/W)\]where $\bc_0$ is the lowest
two-sided cell.  This geometrical equivalence is a Morita
equivalence. This result depends on results of Lusztig-Xi. The
above phenomenon, namely
\[
J_{\be_o}^{\fs} \asymp \mathcal{O}(T^{\vee}/W^\fs) \oplus
\Cset\]
where $\be_0$ is the lowest two-sided cell, is a consequence
of the fact that $H^\fs$ is not simply-connected. This geometrical
equivalence is not a Morita equivalence, nor is it
spectrum-preserving.  It is spectrum-preserving with respect to a
 filtration of length $2$.

 The algebraic variety $(T^\vee\q W^\fs)_{\be_0}$
 has two irreducible components, the primitive ideal space of $J_{\be_0}^\fs$ does not. Hence
the bijection $\mu^\fs$ is not a homeomorphism. This implies in
particular, by using \cite[Theorem~2]{BN}, that there cannot exist
a spectrum preserving morphism from $(T^\vee\q W^\fs)_{\be_0}$ to
$J_{\be_0}^\fs$.
\begin{lem} \label{lem:familySO} The flat family is given by
\[
\mathfrak{X}_{\tau}: (x - \tau^2 y)(1 - \tau^2 xy) = 0\;\;\cup\;\;\{(1,-1)\}.\]
\end{lem}

\begin{proof} 
We now write down all the
{\bf nonunitary} quasicharacters of $T$ which obey the reducibility
conditions~(\ref{twelve}):
\[
\psi^{-1}\chi\otimes\psi\nu\chi,\;\;
\psi\chi\otimes \psi^{-1}\nu^{-1}\chi,\;\;
\psi\chi\otimes\psi\nu^{-1}\chi,\;\;
\psi\chi\otimes\psi\nu\chi,\quad \text{with $\psi \in
\Psi(F^{\times})$.}\]
Note that \begin{enumerate} \item[$\bullet$]
the last two characters are in one $W$-orbit,
namely \[\{\psi\chi\otimes\psi\nu\chi:\psi \in \Psi(F^{\times})\}\]
which, with the same change of variable is
 \[\{\phi\nu^{-1/2}\chi\otimes\phi\nu^{1/2}\chi:\psi \in \Psi(F^{\times})\}\]
Since
\[babab(\phi\nu^{-1/2}\chi\otimes\phi\nu^{1/2}\chi) =
\phi^{-1}\nu^{-1/2}\chi\otimes\phi^{-1}\nu^{1/2}\chi,\] the induced
representations
 \[\{I(\phi\nu^{-1/2}\chi\otimes\phi\nu^{1/2}\chi):\phi \in
\Psi(F^{\times})\}\]
are  parametrized by an algebraic curve
$\mathfrak{C}_1$.  A point on $\mathfrak{C}_1$ has coordinates the
unordered pair $\{z\sqrt{q}, z/\sqrt{q}\}$.
\item[$\bullet$] the first two characters are in another $W$-orbit,
namely \[\{\psi^{-1}\chi\otimes\psi\nu\chi: \psi \in
\Psi(F^{\times})\}\]which with the change of variable $\phi: = \psi\nu^{1/2}$
is
\[\{\phi^{-1}\nu^{1/2}\chi\otimes\phi\nu^{1/2}\chi: \phi
\in \Psi(F^{\times})\}\] Since
\[a(\phi^{-1}\nu^{1/2}\chi\otimes\phi\nu^{1/2}\chi) =
\phi\nu^{1/2}\chi\otimes\phi^{-1}\nu^{1/2}\chi\] the induced
representations\[\{I(\phi^{-1}\nu^{1/2}\chi\otimes\phi\nu^{1/2}\chi):
\phi \in \Psi(F^{\times})\}\] are parametrized by the algebraic
curve $(\bC^{\times})/\bZ/2\bZ$.
  We shall refer to this curve as
$\mathfrak{C}_1'$.  A point on $\mathfrak{C}_1'$ has coordinates the
unordered pair $\{z^{-1}/\sqrt{q},z/\sqrt{q}\}$ with $z \in
\mathbb{C}^{\times}$.
\end{enumerate}

The coordinate ring of $\mathfrak{C}_1$ or $\mathfrak{C}_1'$ is the
ring of balanced Laurent polynomials in one indeterminate $t$. The
map $t + t^{-1} \mapsto x$ then secures an isomorphism \[
\mathbb{C}[t,t^{-1}]^{\mathbb{Z}/2\mathbb{Z}} \cong
\mathbb{C}[x]\] and so
\[
\mathfrak{C}_1 \cong \mathfrak{C}_1' \cong \mathbb{A}^1,\] the
affine line.

The algebraic curves $\mathfrak{C}_1$ and $\mathfrak{C}_1'$
intersect in two points, namely
\[
z_a:=\chi \otimes \nu\chi=(1,q^{-1}),\quad\quad z_c:=\epsilon\chi \otimes
\nu\epsilon\chi=(-1,-q^{-1}).\]
According to the next paragraph, these are the points of length
$4$ and multiplicity $1$.

On the other hand, we note that
\[
I(\epsilon\chi \otimes \chi) :=\Ind_{TU}^G(\epsilon\chi \otimes
\chi) = \pi^{+} \oplus \pi^{-}\] by \cite[$\rG_2$ Theorem]{K},
since $\chi, \epsilon \chi$ are distinct characters of order $2$.
\end{proof}

We define $z_b$ and $z_d$ as
\[z_b:=\nu^{1/2}\psi^{-1}\chi\otimes\nu^{1/2}\psi\chi
=(z^{-1}/\sqrt{q},z/\sqrt{q}),\]
\[z_d:=\nu^{-1/2}\psi\otimes\nu^{1/2}\psi=(z\sqrt{q},z/\sqrt{q}).\]

\begin{lem} \label{lem:cocharacterSO}
Define, for each two-sided cell $\be$ of $\tW_\aff^\fs$,  cocharacters
$h_\be$ as follows:
\[h_{\be}(\tau)  =  (1,\tau^{-2})\quad\text{if $\be = \be_e, \be_1, \be_1'$,}
\quad\text{and
$h_{\be_0} = 1$,}\]
and
define $\pi_\tau(x) = \pi(h_{\be}(\tau)\cdot x)$ for all $x$ in the
$\be$-component.
Then, for all $t \in T^{\vee}/W^\fs$ we have
\[|\pi^{-1}_{\sqrt q}(t)| = |i_{T\subset B}^G(t)|.\]
\end{lem}
\begin{proof}
We observe that
\[I(\nu^{1/2}\psi^{-1}\chi\otimes\nu^{1/2}\psi\chi)=
I(aba(\nu^{1/2}\psi^{-1}\chi\otimes\nu^{1/2}\psi\chi))=
I(\nu^{-1/2}\psi\chi\otimes\nu).\]
Then Lemma~\ref{lem: lengths} gives:
\[|i_{T\subset B}^G(t)|=2\quad\text{if $t=z_b,\;z_d$,}\]
\[
|i_{T\subset B}^G(t)|=4\quad\text{if $t=z_a,\;z_c$.}\]
This leads to the following points of length $4$:
\[
\chi\otimes\nu\chi,\quad \epsilon\chi\otimes\nu\epsilon\chi,\quad
\nu^{-1}\chi\otimes \chi,\quad
\nu^{-1}\epsilon\chi\otimes\epsilon\chi\] Since \[
babab(\nu^{-1}\chi\otimes\chi) = \chi\otimes\nu\chi\]
\[babab(\nu^{-1}\epsilon\chi\otimes\epsilon\chi) =
\epsilon\chi\otimes\nu\epsilon\chi\] this leads to exactly $2$
points in the Bernstein variety $\Omega^{\mathfrak{s}}(G)$ which
pa\-ra\-me\-tri\-ze representations of length $4$, namely $[T,\chi\otimes\nu\chi]_G$ and
$[T,\epsilon\chi\otimes\nu\epsilon\chi]_G$.
The coordinates of these points in the algebraic surface
$\Omega^{\mathfrak{s}}(G)$ are $(1,q^{-1})$ and $(-1,-q^{-1})$.

We will now compute a correcting system of cocharacters (see the paragraph
after Theorem 1.5).   The extended quotient $T^{\vee} \q W^\fs$ is a disjoint
union of $6$ irreducible
components $Z_1$, $Z_2$, $\ldots$, $Z_6$. Our notation is such that $Z_1 =
pt_1$, $Z_2 = pt_2$, $Z_3 \simeq \mathbb{A}^1$, $Z_4 
 \simeq \mathbb{A}^1$, $Z_5 = pt_*$, $Z_6= T^{\vee}/W^\fs$.

 In the following table, the first column comprises $6$ irreducible
components
$X_1, \ldots, X_6$
of $\widetilde{T^{\vee}}$ for which $\rho^{\fii}(X_j) = Z_j, \; j = 1,
\ldots, 6$.  Let $[z_1,z_2]$ denote the image of $(z_1,z_2)$ via the
standard
quotient map $p^\fii\colon T^{\vee} \to T^{\vee}/W^\fs$, so that $[z_1,z_2] =
W^\fs \cdot (z_1,z_2)$.
When the first column itemizes the pairs
$(w,t) \in \widetilde{T^{\vee}}$, the second column
itemizes $p^\fs(h_j(\tau)t)$. The third column itemizes the corresponding
correcting cocharacters.
\[
\begin{array}{lll}
X_1 = (r^3,(1,1))  &  [1,\tau^{-2}] & h_1(\tau) =  (1,\tau^{-2}) \\
X_2 = (r^3,(-1,-1))  &  [-1,-\tau^{-2}] & h_2(\tau) =  (1,\tau^{-2})   \\
X_3 = \{(a,(z,z)): z \in \CC^{\times}\}  &  [z,\tau^{-2}z] & h_3(\tau) =  (1,\tau^{-2}) \\
X_4 = \{(ar^3,(z,z^{-1})): z \in \CC^{\times}\} &  [z,\tau^{-2}z] & h_4(\tau) =  (1,\tau^{-2}) \\
X_5 = (r^3,(1,-1))  &  [1,-1] & h_5(\tau) =  1\\
X_6 = \{(e,(z_1, z_2)): z_1, z_2  \in \CC^{\times}\} & \{[z_1,z_2]:
z_1,z_2 \in \CC^{\times}\} & h_6(\tau) = 1 
\end{array}
\]

It is now clear that  cocharacters can assigned to two-sided cells as
follows:
$h_{\be}(\tau)  =  (1,\tau^{-2})$ if $\be = \be_e, \be_1, \be_1'$,
and $h_{\be_0} = 1$.

We have for instance
\[\pi^{-1}_{\sqrt q}(z_a)=\left\{\rho^\fs(e,(1,q^{-1})),
\rho^\fs(r^3,(1,1)),\rho^\fs(a,(1,1)),\rho^\fs(ar^3,(1,1))\right\},\]
\[\pi^{-1}_{\sqrt q}(z_c)=\left\{\rho^\fs(e,(-1,-q^{-1})),
\rho^\fs(r^3,(-1,-1)),\rho^\fs(a,(-1,-1)),\rho^\fs(ar^3,(-1,-1))\right\}.\]
We have
\[|\pi^{-1}_{\sqrt q}(t)|=2 \quad\text{if $t=z_b,z_d$,}\]
\[|\pi^{-1}_{\sqrt q}(t)|=4 \quad\text{if $t=z_a,z_c$.}\]

Finally, we define $z_*:=\epsilon\chi\otimes\chi=(-1,1)$. Then we have
\[\pi^{-1}_{\sqrt
q}(z_*)=\left\{\rho^\fs(e,(-1,1)),\rho^\fs(r^3,(-1,1))\right\}.\]
Hence
\[|\pi^{-1}_{\sqrt q}(z_*)|=|i_{T\subset B}^G(z_*)|.\] 
\end{proof}

 \begin{lem} \label{lem:bijSO} Part (4) of Theorem \ref{thm:main} is true for the point
\[\fs=[T,\chi\otimes\chi]_G\in \mathfrak{B}(\rG_2)\] where $\chi$
is a ramified quadratic character of $F^\times$.
\end{lem}
\begin{proof} This proof requires a detailed analysis of the
associated  $KL$-pa\-ra\-me\-ters. We recall (see the beginning of
section~\ref{sec:SO})
that \[ H^\fs= \SL(2,
\Cset) \times \SL(2, \Cset)/\langle -\Id,-\Id\rangle.\]
We recall also the beginning of
section~\ref{sec:SO} that the group $H^\fs$ admits $4$
unipotent classes $\be_0$, $\be_1$, $\be_1'$, $\be_0$. We have the
corresponding decomposition of the asymptotic algebra into ideals:
\[
J^\fs = J^\fs_{\be_e} \oplus J^\fs_{\be_1} \oplus J^\fs_{\be_1'}
\oplus J^\fs_{\be_0}.\]

We will write
\[\SL(2, \Cset) \times \SL(2, \Cset) \longrightarrow H^\fs, \quad
(x,y) \mapsto [x,y],\]

\[s_{\nalpha}: =
\left[\begin{array}{cc} \nalpha & 0\\
0 & \nalpha^{-1}
\end{array}\right], \quad\text{for $\nalpha\in\Cset^\times$,}\]
\[u: =
\left[\begin{array}{cc} 1 & 1\\
0 & 1
\end{array}\right].\]
Recall that $T^\vee$ is the standard maximal torus in $H^\fs$. We will
write\[ [s_{\nalpha},s_{\nbeta}] \in T^\vee.\] The group $W^\fs$ is
generated by the element which exchanges $\nalpha$ and
$\nalpha^{-1}$, and the element which exchanges $\nbeta$ and
$\nbeta^{-1}$.

We will now consider separately the $4$ unipotent classes in
$H^{\fs}$.

\subsubsection{Case 1} We consider
\[
S: = [s_{\sqrt q}, s_{\sqrt q}] \in T^\vee, \quad U: = [u,u] \in
H^\fs.\]
These form a semisimple-unipotent-pair, \ie
\[
SUS^{-1} = U^q.\]

We note that the component group of the simultaneous centralizer
of $S$ and $U$ is given by
\[
\Cent(S,U) = \Cent([s_{\sqrt q}, s_{\sqrt q}], [u,u]) = \{[\Id,\Id],[\Id,-\Id]\} =
\Zset/2\Zset.\] We also have:
\[
\Cent([s_{\sqrt q}, s_{- \sqrt q}], [u,u]) = \{[\Id,\Id],[\Id,-\Id]\} =
\Zset/2\Zset.\]

In each case, the associated variety of Borel subgroups is a
point, namely $[b,b]$ where $b$ is the standard Borel subgroup of
$\SL(2,\Cset)$. The $KL$-parameters are given by

\[
([s_{\sqrt q}, s_{\sqrt q}], [u,u], 1), \quad \quad ([s_{\sqrt q},
s_{- \sqrt q}], [u,u],1).\]

These two $KL$-parameters correspond to the ideal $J^\fs_{\be_e}$ in
$J^\fs$, for which we have
\[
J^\fs_{\be_e} \asymp \Cset \oplus \Cset.\]
This is \emph{not} an $L$-packet in the principal series of
$H^\fs$.

\subsubsection{Case 2}  For each $\nalpha \in \Cset^{\times}$, we
have
\[
\Cent([s_{\sqrt q}, s_{\nalpha}], [u,\Id]) = \Zset/2\Zset.\]The
associated variety of Borel subgroups comprises two points, namely
$[b,b]$ and $[b,b^o]$ where $b^o$ is the opposite Borel subgroup,
i.e. the lower-triangular matrices in $\SL(2,\Cset)$.  The
component group $\Zset/2\Zset$ acts on the homology of the two
points as the trivial $2$-dimensional representation. The
$KL$-parameters in this case are given by
\[
([s_{\sqrt q}, s_{\nalpha}], [u,\Id], 1).\]
These parameters correspond to the ideal $J^\fs_{\be_1}$
of $J^\fs$:
\[
J^\fs_{\be_1} \asymp \mathcal{O}(\mathbb{A}^1).\]

\subsubsection{Case 3} We also have the $KL$-parameters

\[
([s_{\nalpha}, s_{\sqrt q}], [\Id,u], 1)\] with $\nalpha \in
\Cset^{\times}$. These parameters correspond to the ideal
$J_{\be'_1}^\fs$ of $J^\fs$:
\[
J_{\be'_1}^\fs \asymp \mathcal{O}(\mathbb{A}^1).\]

\subsubsection{Case 4} We need to consider the component group of
the semisimple-unipotent-pair
\[
([s_{\nalpha},s_{\nbeta}],[\Id,\Id]).\] The component group of this
semisimple-unipotent-pair is trivial unless $\nalpha = \nbeta =\bbi$, where
$\bbi=
\sqrt-1$ denotes a square root of $1$.  In that case we have
\[
\Cent([s_\bbi,s_\bbi],[\Id,\Id]) = \Zset/2\Zset.\] The associated variety of
Borel subgroups of $H^{\fs}$ comprises $4$ points:
\[
[b,b], \quad [b^o, b^o], \quad [b,b^o], \quad [b^o,b]\]  The
generator of the component group $\Zset/2\Zset$ switches $b$ and
$b^o$.  The $4$ points span a vector space of dimension $4$ on
which $\Zset/2\Zset$ acts by switching basis elements as follows:
\[
[b,b] \to [b^o,b^o], \quad \quad [b,b^o] \to [b^o,b]\] Therefore,
$\Zset/2\Zset$ acts as the direct sum of two copies of the regular
representation $1 \oplus \sgn$ of $\Zset/2\Zset$.

We recall that the equivalence relation among the $KL$-parameters
for $H^\fs$  is conjugacy in $H^\fs$. The $KL$-parameters in this
case are

\[
([s_{\nalpha},s_{\nbeta}],[\Id,\Id],1)\] with $\nalpha, \nbeta \in
\Cset^{\times}$, and
\[
([s_\bbi,s_\bbi],[\Id,\Id],\sgn).\]

This corresponds to the ideal $J^\fs_{\be_0} \subset J^\fs$ for which we
have shown in equation~(\ref{eqn:SOJzero}) that
\[
J^\fs_{\be_0} \asymp \mathcal{O}(T^\vee/W^\fs) \oplus \Cset.\]

There is an $L$-packet in the principal series of $H^\fs$ with the
following $KL$-parameters:
\[([s_\bbi,s_\bbi],[\Id,\Id],1), \quad \quad ([s_\bbi,s_\bbi],[\Id,\Id],\sgn).\]
The representations indexed by these $KL$-parameters are tempered. The
corresponding representations of $G$ itself still belong to an $L$-packet
(see the end of subsection~\ref{subsec:indextriples}). These are the
representations denoted $\pi^+$ and $\pi^-$ in the proof of
Lemma~\ref{lem:tempSO}.

Throughout \S 8 we have been using the ring isomorphism
\[
\Cset[X,Y]_0 \cong \Cset[T^{\vee}/W^{\fs}]\] induced by the map
$\zeta\colon T^\vee\to \Aset^2$ defined by
\[\zeta((z_1,z_2)):= (z_1 + z_1^{-1}, z_2 + z_2^{-1}).\] Note that
this isomorphism sends $(\bbi,\bbi)$ to $(0,0) \in \Aset^2$.
This is the unique point in the affine space $\mathbb{A}^2$ which
is fixed under the map $(x,y) \mapsto (-x,-y)$.

Consider the map
\[\quad M_{z,z'}\;\;\mapsto
\cV^\fs_{[s_{\zeta(z)},s_{\zeta(z')}],[U,U],1},\quad\text{if $(z,z')\ne(0,0)$,}\]
\[\{M_{0,0}',M_{0,0}''\}\mapsto\{\cV^\fs_{[s_\bbi,s_\bbi],[\Id,\Id],1},
\cV^\fs_{[s_\bbi,s_\bbi],[\Id,\Id],\sgn}\},\] from the set of
simple $J_\be^\fs$-modules to the subset of $\Irr(G)^\fs$ such
that $[U,U]$ corresponds to the two-sided cell $\be$. This map
induces a bijection which corresponds, at the level of modules of
the Hecke algebra, to the bijection induced by the Lusztig map
$\phi_q$, by the uniqueness property of $\phi_q$.

\end{proof}

 \begin{lem} \label{lem:tempSO} Part (5) of Theorem \ref{thm:main} is true in this
 case.\end{lem}
 \begin{proof}
We start with the list of all those tempered representations of $\rG_2$ which
admit inertial support $\fs$ (see Proposition~\ref{prop:inertialsupports}):
\[
I(\psi_1\chi \otimes \psi_2\chi) \cup
I_{\alpha}(0,\delta(\psi\chi)) \cup \pi(\chi) \cup
\pi(\epsilon\chi) \cup I_{\beta}(0,\delta(\phi\chi))\] where
\[\psi_1: = z_1^{\val_F},\;\; \psi_2: = z_2^{\val_F},\;\;  \psi: =
z^{\val_F}, \;\;\phi: = w^{\val_F}\] are unramified characters of
$F^{\times}$, and \[\pi(\chi)\subset I(\nu\chi\otimes\chi), \quad
\pi(\epsilon\chi) \subset I(\nu\epsilon\chi\otimes \epsilon\chi)\]
are the elements in the discrete series described in
\cite[Prop.~4.1]{M}.

We recall from Lemma~\ref{lem:extSO} that
\[E^\fs\q W^\fs=(E^\fs/W^{\fs}\sqcup pt_*) \sqcup \mathbb{I} \sqcup (pt_1
\sqcup pt_2) \sqcup \mathbb{I}.\]
Then the restriction of $\mu^\fs$ to $\Irr^\fs(G)^\temp$ is
as follows:
\[W^{\fs}
\cdot(z_1,z_2)\mapsto
I(\psi_1\chi_1 \otimes \psi_2\chi_2),\]
unless $z_1 = -1, z_2 = 1$ in which case
\[W^{\fs} \cdot(-1,1) \cup pt_*\mapsto
\pi^{+} \oplus \pi^{-},\]
\[pt_1 \sqcup pt_2 \mapsto
\pi(\chi) \sqcup \pi(\epsilon\chi).\] We note
that $(\psi\chi)^{\vee} = \psi^{-1}\chi$, so that
$I_{\gamma}(0,\delta(\psi\chi)) \cong
I_{\gamma}(0,\delta(\psi^{-1}\chi))$
by \cite{M}, where $\gamma =
\alpha, \beta$.  Finally,
\[z\mapsto
I_{\alpha}(0,\delta(\psi\chi))\quad\text{and}\quad w \mapsto
I_{\beta}(0,\delta(\phi\chi)).\]
\end{proof}

\smallskip

{\it Acknowledgements.}
We would like to thank Goran Mui\'c for sending us several
detailed emails concerning the representation theory of $\rG_2$, and
Nanhua Xi for sending us several detailed emails concerning the
asymptotic algebra $J$ of Lusztig. We also thank the referee for many detailed and constructive comments.

\end{document}